\documentclass[a4paper,11pt,reqno]{article}
\usepackage{latexsym,amsmath,amssymb,amsfonts,amsthm}
\usepackage{a4wide}
\usepackage{graphicx}
\usepackage[colorlinks=true, linkcolor=blue, citecolor=red]{hyperref}

\newtheorem{theorem}{Theorem}[section]
\newtheorem{lemma}[theorem]{Lemma}
\newtheorem{proposition}[theorem]{Proposition}

\newtheorem{remark}[theorem]{Remark}

\numberwithin{equation}{section}

\newcommand{\e}{\varepsilon}
\newcommand{\vphi}{\varphi}
\newcommand{\N}{\mathbb N}

\newcommand{\R}{\mathbb R}
\newcommand{\C}{\mathbb C}

\newcommand{\de}{\,\mathrm{d}}

\newcommand{\Res}{{\rm Res}}

\renewcommand{\Re}{\mathrm{Re}}

\newcommand{\point}{\bigl(\frac{y}{2}\bigr)}

\title{Self-similar gelling solutions for the coagulation equation\\with diagonal kernel}
\author{Marco Bonacini\thanks{\emailmarco} \and Barbara Niethammer\thanks{\emailbarbara} \and Juan J. L. Vel\'{a}zquez\thanks{\emailjuan}}
\date{\uniadd \\[3ex]\today}

\newcommand{\email}[1]{E-mail: \tt #1}
\newcommand{\emailmarco}{\email{bonacini@iam.uni-bonn.de}}
\newcommand{\emailbarbara}{\email{niethammer@iam.uni-bonn.de}}
\newcommand{\emailjuan}{\email{velazquez@iam.uni-bonn.de}}
\newcommand{\uniadd}{\emph{\small University of Bonn, Institute for Applied Mathematics\\ Endenicher Allee 60, 53115 Bonn, Germany}}

\begin{document}

\maketitle

\begin{abstract}
	We consider Smoluchowski's coagulation equation in the case of the diagonal kernel with homogeneity $\gamma>1$. In this case the phenomenon of gelation occurs and solutions lose mass at some finite time. The problem of the existence of self-similar solutions involves a free parameter $b$, and one expects that a physically
	relevant solution (i.e. nonnegative and with sufficiently fast decay at infinity) exists for a single value of $b$, depending on the homogeneity $\gamma$. 
	We prove this picture rigorously for large values of $\gamma$. In the general case, we discuss in detail the behaviour of solutions to the self-similar equation as the parameter $b$ changes.
\end{abstract}

\tableofcontents
\newpage


\section{Introduction} \label{sect:intro}

We consider Smoluchowski's coagulation equation \cite{Smo27}
\begin{equation} \label{eq:smol}
\partial_t f(\xi,t) = \frac12\int_0^{\xi} K(\xi-\eta,\eta)f(\xi-\eta,t)f(\eta,t)\de \eta - \int_0^\infty K(\xi,\eta)f(\xi,t)f(\eta,t)\de \eta,
\end{equation}
which describes binary clustering of particles of size $\xi$ and $\eta$ with a rate prescribed by the rate kernel $K(\xi,\eta)$. Formally, equation \eqref{eq:smol} preserves
the total mass in the system, that is $\int_0^{\infty} \xi f(\xi,t)\de \xi\equiv \mbox{const.}$. However, it is well-known that if the kernel $K$ is for example homogeneous 
and the degree of homogeneity is larger than one, the phenomenon of gelation occurs, that is the loss of mass at finite time. For models of polymers gelation has already been 
predicted in the classical Flory-Stockmayer theory (see e.g. \cite{Ziff80}) via statistical methods. For the discrete analogue of the rate equation \eqref{eq:smol} the first explicit
example of a gelling solution has been given by McLeod \cite{McLeod62a} for the product kernel $K(\xi,\eta)=\xi \eta$, while a solution after the gelation time has been
provided in \cite{LeyTschu81}. Rigorous proofs of the occurence of  gelation for a large range of kernels with  homogeneity  larger than one can be found 
 in \cite{Jeon98,EMP02}. It is in principle conjectured that when
the time approaches the gelation time the solution converges to a self-similar form \cite{Ley03}. However, not much is presently known  about the details of the gelation process apart from the 
case of the solvable product kernel. In this case, it is established in \cite{MP04} that there is a whole one-parameter family of self-similar solutions describing the gelation process.
One of them has exponential decay, the others decay like a power law, and which of them is selected depends on the decay behaviour of the initial data. The proof of the result relies 
on the solvability of the product kernel, that is on the fact that the equation is in  this case explicitly solvable using the Laplace transform. 
For kernels different from the product kernel, however, very little is known about self-similar gelling solutions. 
In \cite{BreFon14} perturbations of the product kernel are considered and the existence of self-similar  solutions is established. Positivity  of these solutions is however not rigorously shown, but a convincing formal argument of this property is given.

Our goal in this paper is to study the question of existence of self-similar gelling solutions  for another special kernel, the so called diagonal kernel $K(\xi,\eta)=\xi^{1+\gamma}\delta(\xi-\eta)$ with homogeneity $\gamma>1$ (here $\delta$ denotes the Dirac delta at the origin, which is homogeneous of degree $-1$). 
Using the identity $\delta(\xi-2\eta)=\frac12\delta(\eta-\frac{\xi}{2})$, the Smoluchowski  equation \eqref{eq:smol} in this case takes the form
\begin{equation} \label{smoldiag}
\partial_t f(\xi,t) = \frac14\Bigl(\frac{\xi}{2}\Bigr)^{\gamma+1}\bigl(f(\xi/2,t)\bigr)^2 - \xi^{\gamma+1} \bigl(f(\xi,t)\bigr)^2\,.
\end{equation}


\subsection{Self-similar solutions}
Self-similar solutions to \eqref{smoldiag} for homogeneities $\gamma>1$ have the form
\begin{equation} \label{ssansatz}
	f(\xi,t) = (T-t)^{ab}F((T-t)^b\xi), \qquad x=(T-t)^b\xi
\end{equation}
depending on two real parameters $a$, $b$, where $T$ is the blow-up time. By plugging this ansatz into the equation one finds that the parameters $a$, $b$ are related by the condition
\begin{equation} \label{parameters}
	b=\frac{1}{1+\gamma-a}\,,
\end{equation}
and the equation for the self-similar profile $F$ is
\begin{equation} \label{ssequation}
	-abF(x) -b x F'(x) = \frac14 \Bigl(\frac{x}{2}\Bigr)^{\gamma+1}\bigl(F(x/2)\bigr)^2 - x^{\gamma+1}\bigl(F(x)\bigr)^2\,.
\end{equation}
Notice that, in order to have that the mean cluster size $s(t)=(T-t)^{-b}$ tends to infinity as $t$ approaches the gelation time, we impose the constraint $b>0$ (or equivalently $a< 1+\gamma$).

The problem has a free parameter $a$ (or $b$), which gives the power-law behaviour of the self-similar profile at the origin. Indeed, if we require that at the gelpoint $t=T$ the self-similar solution $f(x,T)$ is finite and nonvanishing, then this can only happen if
\begin{equation} \label{ssdecay0}
	F(x)\sim c_0 x^{-a} \qquad\text{as }x\to0^+
\end{equation}
(in the sense $\lim_{x\to0^+}x^{a}F(x)=c_0\in(0,\infty)$). The multiplicative constant $c_0$ can be normalized by rescaling the solution, so that we consider without loss of generality $c_0=1$.
At the gelation time the self-similar solution $f$ has the power-law behaviour
\begin{equation}
	f(\xi,T^-) \sim \lim_{t\to T^-}(T-t)^{ab} \bigl((T-t)^b\xi\bigr)^{-a} = \xi^{-a}\,.
\end{equation}


\subsection{The shooting problem} \label{subsect:shooting}
For the analysis of \eqref{ssequation} it is convenient to rescale the self-similar profile $F$: setting
\begin{equation} \label{variables}
	\Phi(x)=x^{\gamma+1}F(x)\,,
\end{equation}
we obtain from \eqref{ssequation} and \eqref{ssdecay0} that $\Phi$ solves
\begin{equation} \label{equation}
	\begin{cases}
	bx\Phi'(x) = \Phi(x) - 2^{\gamma-1}\bigl(\Phi(x/2)\bigr)^2 + \bigl(\Phi(x)\bigr)^2\,,\\
	\Phi(x) \sim x^{\frac{1}{b}} \qquad\text{as } x\to0^+
	\end{cases}
\end{equation}
(recall also \eqref{parameters}). It is conjectured in \cite{Ley05} that, for a given $\gamma>1$, there is a unique value of the free parameter $b$ such that \eqref{equation} has a physically relevant solution. By physically relevant we mean that the profile $F(x)$ is nonnegative and decays exponentially as $x\to\infty$; in this case a function with the self-similar form \eqref{ssansatz} describes a solution of \eqref{smoldiag} which decays exponentially for $\xi>>(T-t)^{-b}$ and behaves as a power law for $\xi<<(T-t)^{-b}$. This is the analogon of the gelling transition which takes place for the product kernel, see \cite{McLeod62a}.

While there is sound numerical work supporting this expectation, a rigorous proof of such statement seems far from being easily achievable. Leyvraz \cite{Ley05} suggested that a shooting argument could lead to the desired result. Indeed, \eqref{equation} has the explicit power-law solution
\begin{equation}\label{phi0}
	\Phi_0(x) = x^{\frac{1}{b_0}} \qquad\qquad\text{for}\quad b_0=\frac{2}{\gamma-1}
\end{equation}
(notice that, according to \eqref{ssansatz} and \eqref{parameters}, this solution corresponds to the stationary self-similar solution $f(x,t)=x^{-a_0}$, $a_0=\frac{\gamma+3}{2}$, to Smoluchowski's equation \eqref{smoldiag}). It has been observed numerically, and we present below a rigorous proof of this fact (see Proposition~\ref{prop:btob0}), that the solutions corresponding to values of $b>b_0$ sufficiently close to $b_0$ cross to negative values, and therefore are not physically relevant. Moreover, the constant function
\begin{equation}\label{phiinf}
\Phi_\infty(x)\equiv\frac{1}{2^{\gamma-1}-1}
\end{equation}
is also a solution to the first equation in \eqref{equation}, for every value of $b$; solutions to \eqref{equation} for large values of $b$ approach this constant value as $x\to\infty$ (Proposition~\ref{prop:btoinf}). We will not further investigate in this paper the features of these solutions for which $\Phi(x)$ approaches a constant as $x\to\infty$; we only observe that they correspond to solutions of \eqref{smoldiag} in the self-similar form \eqref{ssansatz} with the asymptotic behaviour
\begin{equation} \label{pw}
f(\xi,t)\sim\frac{\Phi_\infty}{(T-t)\xi^{\gamma+1}} \qquad\text{ as }\xi\to\infty.
\end{equation}
It is worth to remark that these solutions have time-dependent fat tails, and in this respect they have analogies with those constructed in \cite{BNV1,BNV2} for kernels with homogeneity $\gamma\leq1$.
It would be interesting to clarify if there are solutions of the full evolution problem \eqref{smoldiag} (not necessarily of self-similar form) with the asymptotics $f(\xi,t) \sim a(t)\xi^{-(\gamma+1)}$ as $\xi \to \infty$, and having also finite mass $\int_0^\infty \xi f(\xi,t)\de\xi<\infty$ for $0\leq t< T$. Notice that this finite mass condition is not satisfied by the self-similar solutions obtained in this paper and having the asymptotics \eqref{pw}, because they have a non-integrable singularity at $\xi=0$.

At the border between these two different behaviours it is expected that a unique value of the shooting parameter $\bar{b}\in(b_0,\infty)$ exists so that the corresponding solution has exponential decay and remains positive.

However, a deeper investigation shows that the structure of the equation is more complex. Indeed the stability analysis of the constant solution $\Phi_\infty$ (see Section~\ref{subsect:stab}) reveals that there is a critical value $b_*$ of the parameter $b$ such that $\Phi_\infty$ is stable only for $b>b_*$, and unstable otherwise. Solutions to \eqref{equation} for $b<b_*$ develop oscillations, which have already been numerically observed. We will investigate the features of these oscillating solutions in Section~\ref{sect:oscill}. We did not explore systematically the region $b<b_*$. It looks feasible to prove the existence of periodic solutions for $b$ in a neighbourhood of $b_*$ by a Hopf bifurcation argument. For smaller values of $b$, it might be a priori possible to have secondary Hopf bifurcations and very complicated oscillatory dynamics. It would be worth to get a deeper insight about this question by means of numerical simulations.

The value of $b_*$ can be determined explicitly and one finds $b_0<b_*$ for every $\gamma$. We expect therefore that the critical parameter $\bar{b}$ is to be found in the interval $(b_0,b_*)$, and that it separates two unphysical behaviours of the solutions: on the one hand, for $b<\bar{b}$, solutions which become negative; on the other, for $b>\bar{b}$, oscillating solutions with $\liminf_{x\to\infty}\Phi(x)>0$. A standard shooting argument seems however to be difficult to perform due to the instability of this second behaviour.

\medskip
In the limit regime $\gamma\to\infty$ we can actually provide a rigorous proof of the previous picture. Indeed, by rescaling the solution as
\begin{equation*}
\phi(x):= \sigma x^{-\frac{1}{b}}\Phi\Bigl(\frac{x}{\sigma^b}\Bigr)\,, \qquad\qquad\text{with }\sigma=2^{\gamma-1-\frac{2}{b}},
\end{equation*}
one finds that the equation solved by $\phi$ is
\begin{equation*}
\begin{cases}
bx\phi'(x)= -x^{\frac1b}\bigl(\phi(x/2)\bigr)^2 + \sigma^{-1}x^\frac1b\bigl(\phi(x)\bigr)^2,\\
\phi(0)=1.
\end{cases}
\end{equation*}
In the limit $\gamma\to\infty$ one has $\sigma\to\infty$ and the second term on the right-hand side is negligible; therefore we can consider the approximation of the equation
\begin{equation*}
bx\phi'(x) = -x^{\frac1b}\bigl(\phi(x/2)\bigr)^2\,.
\end{equation*}
It is immediately checked that the exponential $\phi(x)=e^{-x}$ is an explicit solution for the critical value $b=1$. By linearizing around this solution, it is possible to prove rigorously via the Implicit Function Theorem that for every sufficiently large $\gamma$ there is a unique value $\bar{b}$ in a neighborhood of $b=1$ such that the corresponding solution to \eqref{equation} is positive and has exponential decay at infinity. The proof will be given in Section~\ref{sect:gammatoinf}.

\medskip
We summarize the above discussion and the results of the paper with a cartoon of the phase diagram of the behaviour of solutions to \eqref{equation} depending on the values of the homogeneity $\gamma>1$ and of the free parameter $b$, see Figure~\ref{fig:summary}. On the line corresponding to the value $b_0=\frac{2}{\gamma-1}$ the equation has the explicit solution $\Phi_0$, see \eqref{phi0}. According to the previous discussion, we are interested only in values of $b>b_0$, that is in the region below the line in the picture. A second value $b_1=b_1(\gamma)$ of the parameter is found so that for $b\in(b_0,b_1)$ the solution becomes negative, see Section~\ref{sect:btob0}. A third line $b=b_*(\gamma)$ delimits the region where the constant solution $\Phi_\infty$, see \eqref{phiinf}, is stable; this is discussed in Section~\ref{sect:oscill}. The three lines are ordered, as for every value of $\gamma$ one has $b_0<b_1<b_*$. The asymptotic analysis in Section~\ref{sect:gammato1} provides qualitative information about the solutions for values of $\gamma$ and $\frac1b$ in a neighborhood of the point $(1,0)$. Finally, in the regime $\gamma\to\infty$ we have a rigorous proof of the existence of a critical parameter $\bar{b}=\bar{b}(\gamma)$, with $\bar{b}(\gamma)\to1$ as $\gamma\to\infty$, for which the solution is positive and exponentially decaying.
\begin{figure}
	\centering
	\includegraphics[scale=1.5]{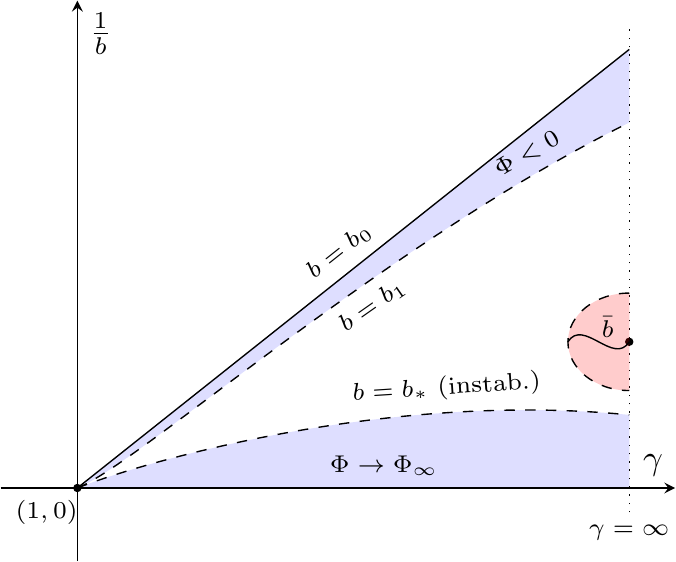}
	\caption{Cartoon of the behaviour of solutions to \eqref{equation} depending on the values of $\gamma$ and $b$.}
	\label{fig:summary}
\end{figure}


\section{Well-posedness and change of sign} \label{sect:btob0}

It will be often convenient to work in another set of variables: we set
\begin{equation} \label{variables2}
	\Phi(x) = y H(y), \qquad y=x^\frac{1}{b}\,,
\end{equation}
with the profile $H$ solving
\begin{equation} \label{equation2}
	\begin{cases}
	H'(y) = -\sigma\bigl(H(2^{-\frac1b}y)\bigr)^2 + \bigl(H(y)\bigr)^2\,, & \qquad\qquad \sigma=2^{\gamma-1-\frac{2}{b}}, \\
	H(0)=1.
	\end{cases}
\end{equation}
One of the advantages of working with these variables is that the solution $H$ turns out to be monotone and analytic in a neighborhood of the origin, see Lemma~\ref{lem:wellpos}. Notice that the two solutions $\Phi_0$ and $\Phi_\infty$, defined in \eqref{phi0} and \eqref{phiinf} respectively, are transformed into
\begin{equation*}
H_0(y) \equiv 1 \qquad\text{ and }\qquad H_\infty(y)=\frac{1}{(2^{\gamma-1}-1)y}\,.
\end{equation*}
The two generic, unphysical behaviours which are expected for solutions to \eqref{equation}, namely $\Phi$ changing sign and $\liminf_{x\to\infty}\Phi(x)>0$, correspond in these variables to $H$ changing sign and to a power-law decay $H(y)\sim\frac{1}{y}$ as $y\to\infty$.

As a preliminary step in the analysis we show the well-posedness of the equation. This is also observed in \cite{Ley05}, and we reproduce here the argument for the reader's convenience.

\begin{lemma}[Well-posedness and monotonicity] \label{lem:wellpos}
	The equation \eqref{equation2} has a unique analytic local solution in a neighborhood of $y=0$. Moreover, assuming that $b>b_0=\frac{2}{\gamma-1}$, the solution remains well-defined and monotonically decreasing as long as $H(y)>0$, and satisfies the bound
	\begin{equation} \label{eq:wellpos}
	H(y)\leq \frac{1}{1+(\sigma-1)y}\,.
	\end{equation}
\end{lemma}

\begin{proof}
	In order to prove local existence in a neighborhood of the origin, one writes $H$ as a power series around $y=0$, $H(y) = \sum_{n=0}^\infty a_n y^n$, with $a_0=1$, and easily determines a recursive formula for the $a_n$'s:
	\begin{equation*}
	a_{n+1} = \frac{1}{n+1} \Bigl(1-\frac{\sigma}{2^{\frac{n}{b}}}\Bigr)\sum_{k=0}^na_ka_{n-k}\,, \qquad n=0,1,\ldots
	\end{equation*}
	As the series converges in a finite interval (to see this, one can for instance prove by induction that $|a_n|\leq c^n$, where $c=\sup_n|1-\sigma2^{-n/b}|$), we obtain an analytic solution to \eqref{equation2} in an interval $[0,y_0]$ for some $y_0>0$.

	We now observe that the solution $H$ is monotonically decreasing as long as it exists and remains positive. Indeed, we have $H'(0)=-(\sigma-1)<0$, and therefore $H$ is initially decreasing. Moreover, if $\bar{y}$ is the first point at which $H'(\bar{y})=0$ and $H$ is positive at $\bar{y}$, then $H$ is monotonically decreasing for $y<\bar{y}$ and by using the equation
	$$
	0= H'(\bar{y}) \leq -(\sigma-1)\bigl(H(2^{-\frac1b}\bar{y})\bigr)^2 <0\,,
	$$
	which is a contradiction.
	
	Once we have local existence in an interval $[0,y_0]$ near the origin, the equation \eqref{equation2} can be regarded as an ordinary differential equation and a recursive application of standard existence and uniqueness results give that the solution can be uniquely continued, unless it becomes negative. Indeed the monotonicity implies that the solution remains bounded as long as it stays positive.

	Finally, to prove the bound in the statement it is sufficient to observe that by monotonicity we have
	\begin{align*}
	H'(y) &= - (\sigma-1)\bigl(H(2^{-\frac{1}{b}}y)\bigr)^2 + \bigl(H(y)\bigr)^2 - \bigl(H(2^{-\frac{1}{b}}y)\bigr)^2 \\
	& \leq - (\sigma-1)\bigl(H(2^{-\frac{1}{b}}y)\bigr)^2 \leq - (\sigma-1)\bigl(H(y)\bigr)^2\,,
	\end{align*}
	that is, $(\frac{1}{H(y)})' \geq (\sigma-1)$. Therefore \eqref{eq:wellpos} follows by integration.
\end{proof}

We show in the following proposition that for values of $b$ larger but close to $b_0=\frac{2}{\gamma-1}$ the solution crosses to negative values and is therefore not a physical solution.

\begin{proposition}[Change of sign for $b\sim b_0$] \label{prop:btob0}
	There exists $b_1>b_0$, depending on $\gamma$, such that for every $b\in(b_0,b_1)$ the solution $H$ to \eqref{equation2} becomes negative.
\end{proposition}

\begin{proof}
Since $b\to b_0^+$ corresponds to $\sigma\to1^+$, it is convenient to introduce a small parameter $\e>0$ by setting $\sigma-1=2\e$ (so that also $b$ depends on $\e$). The lemma will be proved by showing that $H$ has to change sign, provided $\e$ is small enough.
We linearize around $H(0)=1$: we set $H(y)=1+M(y)$, with $M(0)=0$ and
\begin{equation} \label{btob01}
M'(y) = -2\e + 2M(y) -2\sigma M(2^{-\frac{1}{b}}y) - \sigma\bigl(M (2^{-\frac{1}{b}}y)\bigr)^2 + \bigl(M(y)\bigr)^2\,.
\end{equation}
We now divide the proof into four steps.

\medskip\noindent
\textit{Step 1.}
We first consider the linearized equation
\begin{equation} \label{btob02}
L'(y) = 1 + 2L(y) -2L(2^{-\frac1b}y)\,,\qquad L(0)=0\,.
\end{equation}
By writing $L$ in power series, $L(y) = \sum_{n=1}^\infty a_n y^n$, it is straightforward to determine a recursive relation for the coefficients $a_n$ which yields the following expression for the solution $L$, defined for all $y$,
\begin{equation*}
L(y) = \sum_{n=1}^\infty \Bigl(\frac{2^{n-1}}{n!}\Bigr)c_n(b)y^n\,, \qquad c_1(b)=1,\quad c_n(b):=\prod_{k=1}^{n-1}(1-2^{-\frac{k}{b}})\,.
\end{equation*}
Notice that $(c_n(b))_n$ is a decreasing sequence with
\begin{equation*}
c_n(b) \to c_*(b) := \prod_{k=1}^\infty (1-2^{-\frac{k}{b}}) \in(0,1) \qquad\text{as }n\to\infty.
\end{equation*}
Therefore the solution to the linearized equation \eqref{btob02} satisfies the bounds
\begin{equation} \label{btob03}
\frac{c_*(b)}{2}(e^{2y}-1) \leq L(y) \leq \frac12 (e^{2y}-1)\,.
\end{equation}

\medskip\noindent
\textit{Step 2.}
We now go back to the solution $M$ to \eqref{btob01} and we write
\begin{equation} \label{btob04}
M(y) = -2\e L(y) + S(y)\,, \qquad S(0)=0\,,
\end{equation}
where $S(y)$ solves
\begin{align*} 
S'(y)
& = 2S(y) - 2 S(2^{-\frac1b}y) -4\e M(2^{-\frac1b}y) - \sigma\bigl(M(2^{-\frac1b}y)\bigr)^2 +\bigl(M(y)\bigr)^2 \\
& =: 2S(y) - 2 S(2^{-\frac1b}y) +R(y)\,.
\end{align*}
This equation satisfies a maximum principle: indeed, suppose that $\bar S$ solves
\begin{equation} \label{btob04bis}
\bar{S}'(y) = 2\bar S(y) - 2\bar S(2^{-\frac1b}y) + \bar R(y)\,,
\qquad \bar S(0)=0\,,
\end{equation}
with a positive source $\bar R > 0$. Then $\bar S$ is positive and increasing: indeed, at $y=0$ one has $S'(0)=0$ and $S''(0)>0$, therefore $\bar S(y)>\bar S(2^{-\frac1b}y)$ for $y$ sufficiently small; by plugging this inequality into the equation \eqref{btob04bis} one finds that the derivative remains positive for larger $y$. Hence $\bar S$ is increasing and $\bar S>0$. Moreover
\begin{equation*}
\bar{S}'(y) \leq 2\bar S(y) + \bar{R}(y)\,,
\end{equation*}
which yields, as $\bar S(0)=0$,
\begin{equation} \label{btob04ter}
\bar S(y) \leq \int_0^y e^{2(y-z)}\bar{R}(z)\de z\,.
\end{equation}

\medskip\noindent
\textit{Step 3.}
We now claim that we can find uniform constants $\delta>0$ and $K>1$ such that
\begin{equation} \label{btob05}
|M(y)| \leq \e K e^{2y}
\qquad\text{for every } y\in(0,y_*),
\end{equation}
where $y_*>0$ is the point such that $\e e^{2y_*}=\delta$.

We prove the previous estimate by a continuation argument. Suppose that \eqref{btob05} is valid for every $y\in(0,\bar{y})$, for some $\bar{y}<y_*$. Then for every $y\leq\bar{y}$
\begin{equation} \label{btob06}
\begin{split}
|R(y)|
& \leq 4\e |M(2^{-\frac1b}y)| + \sigma \bigl( M(2^{-\frac1b}y)\bigr)^2 + \bigl(M(y)\bigr)^2  \\
& \leq 4\e^2K e^{2^{1-1/b}y} + \sigma\e^2 K^2e^{2^{2-1/b}y} + \e^2 K^2 e^{4y}
\leq 10\e^2K^2e^{4y}\,.
\end{split}
\end{equation}
Define now $\bar R(y):= 10\e^2K^2e^{4y}$, and let $\bar S$ be the solution to \eqref{btob04bis}. Notice that, as $\bar R>0$, we have by \eqref{btob04ter}
\begin{equation*}
\bar{S}(y) \leq \int_0^y e^{2(y-z)}\bar R(z)\de z
= 10\e^2K^2e^{2y}  \int_0^{y} e^{2z}\de z
\leq 5\e^2K^2e^{4{y}}\,.
\end{equation*}
Since $|R(y)|\leq\bar{R}(y)$ for every $y\in(0,\bar{y})$ by \eqref{btob06}, the discussion in the second step implies that for every $y\leq\bar{y}$
\begin{equation} \label{btob06bis}
|S(y)| \leq \bar S(y) \leq 5\e^2K^2e^{4{y}}\,.
\end{equation}
By inserting this inequality (computed at $\bar{y}$) in \eqref{btob04}, and recalling also \eqref{btob03}, we finally obtain
\begin{align*}
|M(\bar{y})|
&\leq 2\e L(\bar y) + |S(\bar y)|
\leq \e(e^{2\bar{y}}-1) + 5\e^2K^2e^{4{\bar y}}
\leq \bigl( 1 + 5\delta K^2 \bigr) \e e^{2\bar{y}}\,,
\end{align*}
where we used the fact that $\bar{y}\leq y_*$ and the definition of $y_*$ in the last inequality. It is clear that we can choose any constant $K>2$ and a sufficiently small $\delta>0$ (depending only $K$) such that
\begin{equation} \label{btob07}
|M(\bar y)| \leq \frac{\e K}{2}e^{2\bar{y}}\,.
\end{equation}
For this choice of $K$ and $\delta$ we have therefore proved that, assuming that \eqref{btob05} is valid for every $y\in(0,\bar{y})$, for some $\bar{y}<y_*$, then \eqref{btob07} holds at $\bar{y}$. Since \eqref{btob05} is obviously true for $y$ small enough, as $M(0)=0$, a continuation argument yields that \eqref{btob05} actually holds in the full interval $(0,y_*)$, as claimed.

\medskip\noindent
\textit{Step 4.}
We are now in position to conclude the proof of the lemma. We have by \eqref{btob04}, \eqref{btob03}, \eqref{btob06bis}, and by definition of the point $y_*$,
\begin{align*}
H(y_*)
& = 1+M(y_*)
= 1 -2\e L(y_*) + S(y_*) \\
& \leq 1 - \e c_*(b) (e^{2y_*}-1) + 5\e^2K^2e^{4y_*} \\
& = 1 - \delta c_*(b) + \e c_*(b) + 5\delta^2K^2\,.
\end{align*}
Observing that the constant $c_*(b)$ is uniformly bounded away from 0 for $b$ in a neighborhood of $b_0$, by possibly choosing a smaller $\delta>0$ we obtain from this inequality that for every $\e>0$ sufficiently small one has
\begin{equation} \label{btob08}
H(y_*) \leq 1-\delta_*
\end{equation}
for some fixed $\delta_*\in(0,1)$. In turn, by the monotonicity proved in Lemma~\ref{lem:wellpos} we also have that $H(y)\leq  1-\delta_*$ for $y>y_*$, as long as $H$ remains positive.

Consider now any point $y\in(y_*,y_*+\frac{2}{\delta_*})$. For every such $y$ we have
\begin{equation*}
\frac{y}{2^{\frac1b}} \leq \frac{y_*}{2^{\frac1b}} + \frac{2}{2^{\frac1b}\delta_*}
= \frac{\ln(\delta/\e)}{2^{1+\frac1b}} + \frac{2}{2^{\frac1b}\delta_*} < \frac12\ln(\delta/\e)
\end{equation*}
provided $\e$ is small enough. Hence $2^{-\frac1b}y<y_*$, so that we can use the estimate \eqref{btob06bis} at the point $2^{-\frac1b}y$. This estimate, together with \eqref{btob04}, \eqref{btob03} and the definition of the point $y_*$, yields
\begin{align*}
\big|H(2^{-\frac1b}y)-1\big|
& = |M(2^{-\frac1b}y)|
\leq 2\e L(2^{-\frac1b}y) + |S(2^{-\frac1b}y)| \\
& \leq \e(e^{2^{1-1/b}y}-1) + 5\e^2K^2e^{2^{2-1/b}y} \\
& \leq \e e^{2^{1-1/b}y_*} \exp\Bigl(\frac{2^{2-1/b}}{\delta_*}\Bigr) + 5\e^2K^2e^{2^{2-1/b}y_*} \exp\Bigl(\frac{2^{3-1/b}}{\delta_*}\Bigr) \\
& = \e^{1-2^{-1/b}}\delta^{2^{-1/b}}\exp\Bigl(\frac{2^{2-1/b}}{\delta_*}\Bigr) + 5K^2 \e^{2(1-2^{-1/b})}\delta^{2^{1-1/b}}\exp\Bigl(\frac{2^{3-1/b}}{\delta_*}\Bigr)
\end{align*}
for every $y\in(y_*,y_*+\frac{2}{\delta_*})$. Therefore
\begin{equation} \label{btob09}
\big|H(2^{-\frac1b}y)-1\big| \leq \omega_{\delta,\delta_*}(\e) \qquad\text{for every }y\in\Bigl(y_*,y_*+\frac{2}{\delta_*}\Bigr)
\end{equation}
for some function $\omega_{\delta,\delta_*}(\e)\to0$ as $\e\to0^+$.

By combining \eqref{btob08} and \eqref{btob09} we have from \eqref{equation2}
\begin{align*}
H'(y) &= -(1+2\e)\bigl( H(2^{-\frac1b}y) \bigr)^2 + \bigl(H(y)\bigr)^2 \\
& \leq -(1+2\e)(1-\omega_{\delta,\delta_*}(\e))^2 + (1-\delta_*)^2 \\
& \leq -\delta_* + \tilde{\omega}(\e)\,,
\end{align*}
with $\tilde{\omega}(\e)\to0$ as $\e\to0^+$, for every $y\in(y_*,y_*+\frac{2}{\delta_*})$, provided that $H$ remains positive in that interval.
By choosing $\e$ small enough, we then have $\frac{\de H}{\de y}(y) \leq -\frac12\delta_*$ in the interval $(y_*,y_*+\frac{2}{\delta_*})$, and this together with \eqref{btob08} implies that $H$ has to become negative in this interval.
\end{proof}


\section{Stability and oscillating solutions} \label{sect:oscill}

We discuss in this section the stability of the constant solution $\Phi_\infty\equiv\frac{1}{2^{\gamma-1}-1}$ to the first equation in \eqref{equation} (corresponding to $H_\infty(y)=\frac{\Phi_\infty}{y}$ in the variables \eqref{variables2}). For this purpose it is convenient to introduce a third set of variables, by setting
\begin{equation} \label{variables3}
	\vphi(z) = yH(y) = \Phi(x), \qquad e^{z}= y = x^{\frac1b}, \qquad z=\ln y = \frac1b\ln x,
\end{equation}
with the profile $\vphi:\R\to\R$ solving the nonlinear delay equation
\begin{equation} \label{equation3}
	\begin{cases}
	\vphi'(z) = \vphi(z) -2^{\gamma-1}\bigl(\vphi(z-d)\bigr)^2 + \bigl(\vphi(z)\bigr)^2\,, & \qquad\qquad d=\frac{\ln2}{b}, \\
	\vphi(z) \sim e^z \qquad\text{as }z\to-\infty\,.
	\end{cases}
\end{equation}
We will first show in Proposition~\ref{prop:stability} that the constant $\Phi_\infty$ is asymptotically stable, provided that the free parameter $b$ is larger than a critical value $b_*$, which is explicitly determined in terms of the homogeneity $\gamma$. A consequence of this result is that, for large values of $b$, the solution $\vphi$ to \eqref{equation3} converges to the constant value $\Phi_\infty$ as $z\to\infty$, see Proposition~\ref{prop:btoinf}. Finally, in Section~\ref{subsect:oscill} we discuss oscillating solutions which appear numerically for intermediate values of $b$ and are related to the instability of $\Phi_\infty$. In particular we present an iteration argument describing in details the structure of such solutions in the limit case $\gamma=\infty$.


\subsection{Stability analysis} \label{subsect:stab}
We first study the asymptotic stability of the constant solution $\Phi_\infty$.

\begin{proposition}[Stability of the constant solution]\label{prop:stability}
	Assume that $b>b_*$, where $b_*=b_*(\gamma)$ is the critical value
	\begin{equation} \label{b*}
	b_* = \frac{2^\gamma\ln2 \sqrt{1-\bigl(\frac12+\frac{1}{2^\gamma}\bigr)^2}}{(2^{\gamma-1}-1)\arccos\bigl(\frac12+\frac{1}{2^\gamma}\bigr)}\,.
	\end{equation}
	Then the constant solution $\Phi_\infty=\frac{1}{2^{\gamma-1}-1}$ is asymptotically stable, in the following sense: there exists $\delta>0$ such that, given any $z_0\in\R$ and $\vphi_0\in C^0([z_0-d,z_0])$ with $\|\vphi_0-\Phi_\infty\|_{\infty}<\delta$, the initial value problem
	\begin{equation*}
	\begin{cases}
	\vphi'(z) = \vphi(z) -2^{\gamma-1}\bigl(\vphi(z-d)\bigr)^2 + \bigl(\vphi(z)\bigr)^2 & \text{for } z\geq z_0, \\
	\vphi(z) = \vphi_0(z) & \text{for }z\in[z_0-d,z_0]
	\end{cases}
	\end{equation*}
	has a unique, global solution which converges exponentially to $\Phi_\infty$ as $z\to\infty$.
\end{proposition}

\begin{proof}
If we plug the ansatz $\vphi(z)=\Phi_\infty + \phi(z)$ into the first equation of \eqref{equation3} we find
\begin{equation} \label{stab0}
\phi'(z) = \frac{\theta+1}{\theta-1}\phi(z) - \frac{2\theta}{\theta-1}\phi(z-d) - \theta\bigl(\phi(z-d)\bigr)^2 + \bigl(\phi(z)\bigr)^2\,,
\end{equation}
where we set $\theta=2^{\gamma-1}$ to lighten the notation. In terms of $\phi$, the goal is to show the asymptotic stability of the trivial solution to \eqref{stab0}.

The conclusion follows by an application of standard results in the theory of delay differential equations. It is indeed well known that the trivial solution to the linearized equation
\begin{equation*}
G'(z) = \frac{\theta+1}{\theta-1} G(z) - \frac{2\theta}{\theta-1}G(z-d)
\end{equation*}
is uniformly asymptotically stable if all the roots of the corresponding characteristic equation
\begin{equation} \label{charequation}
\lambda +\frac{2\theta}{\theta-1}e^{-d\lambda} -\frac{\theta+1}{\theta-1}=0
\end{equation}
have negative real part, see for instance \cite[Theorem~28B]{Dri77} and references therein. In view of Lemma~\ref{lem:stability} below, this is the case if $b>b_*$. Therefore in this case we are in the position to apply \cite[Theorem~34E and Corollary~34F]{Dri77} which yield the uniform asymptotic stability of the trivial solution to the nonlinear equation \eqref{stab0}.
\end{proof}

\begin{lemma}[Roots of the characteristic equation]\label{lem:stability}
	All the roots $\lambda\in\C$ of the characteristic equation \eqref{charequation} have negative real part if and only if $b>b_*$, where $b_*$ is defined in \eqref{b*}, and $d=\frac{\ln2}{b}$.
\end{lemma}

\begin{proof}
By setting $\tilde \lambda:=\frac{\theta-1}{2\theta}\lambda$, $\tilde d:=\frac{2\theta}{\theta-1}d$, $\tilde{\sigma}:=\frac{\theta+1}{2\theta}$, the problem is equivalent to study for which values of $\tilde{A}$ all the roots $\tilde \lambda\in\C$ of the equation
\begin{equation*}
e^{-\tilde{d}\tilde\lambda} + \tilde\lambda - \tilde{\sigma} = 0
\end{equation*}
have negative real part. Notice that $\tilde{\sigma}=\frac12+\frac{1}{2^\gamma}\in(\frac12,1)$.

Define $F_{\tilde{d}}(\lambda):= e^{-\tilde{d}\lambda} + \lambda - \tilde{\sigma}$. We make use of the Argument Principle in order to study the roots of $F_{\tilde{d}}$ in the complex plane. We consider the closed curve $\Gamma$ in the complex plane surrounding an half disk of radius $R>>1$ in the half plane $\{\Re(z)>0\}$ with flat part on the imaginary axis: $\Gamma=\Gamma_1\cup\Gamma_2$, with
\begin{equation*}
\Gamma_1 = \{it : t\in[-R,R] \}\,,
\qquad
\Gamma_2 = \{Re^{-it} : t\in(-\pi/2,\pi/2)\}\,.
\end{equation*}
Let $\Sigma:=F_{\tilde{d}}(\Gamma)$ be the image of $\Gamma$ under the map $F_{\tilde{d}}$. We shall now count the number of turns of $\Sigma$ around the origin: if such number is zero, we can conclude that there are no roots of $F_{\tilde{d}}$ in the region enclosed by $\Gamma$ (and since $R$ is arbitrary there are no roots in the whole half plane $\{\Re(z)>0\}$).

\begin{figure}
	\centering
	\includegraphics{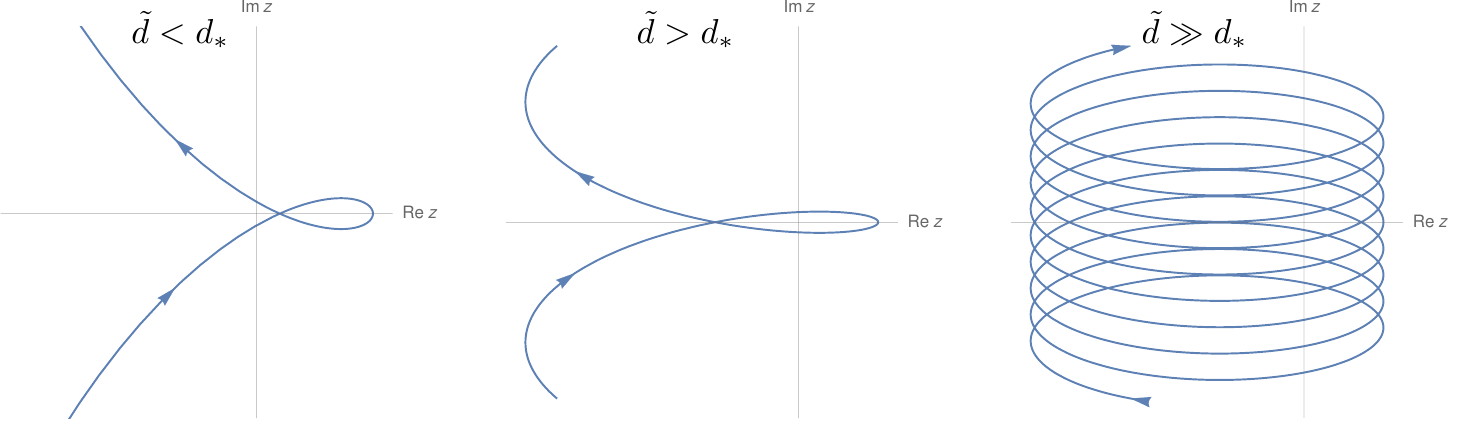}
	\caption{The curve $\Sigma_1$ in the proof of Lemma~\ref{lem:stability}, see \eqref{curvestab}, for different values of the parameter $\tilde{d}$. For $\tilde{d}$ smaller than the critical parameter $d_*$ the curve does not surround the origin (left panel); for $\tilde{d}>d_*$ the curve makes a loop around the origin and roots with positive real part appear (center); for large values of $\tilde{d}$ the number of turns around the origin increases, corresponding to more instability modes (right).}
	\label{fig:argument}
\end{figure}

It is easily seen that the path $F_{\tilde{d}}(\Gamma_2)$, for large values of $R$, is approximately a half-circle of radius $R$ centered at the point $-\tilde\sigma$ on the real axis. The curve $F_{\tilde{d}}(\Gamma_1)$ has instead the parametric expression
\begin{align} \label{curvestab}
\Sigma_1 = \{(-\tilde{\sigma}+\cos(\tilde{d}t)) + i (t - \sin(\tilde{d}t)) : t\in[-R,R]\}\,.
\end{align}
For $\tilde{d}\leq1$ the imaginary part is monotone increasing in $t$, and this clearly implies that the curve $\Sigma$ does not surround the origin. However, as $\tilde{d}$ increases past $1$, the curve starts to make a little loop, eventually surrounding the origin for $\tilde{d}$ larger than a critical value $d_*$ (see Figure~\ref{fig:argument}). It is a simple exercise to check that the curve $\Sigma_1$ does not make any turn around the origin if and only if
$$
\tilde{d} < \frac{\arccos\tilde{\sigma}}{\sin\arccos\tilde{\sigma}} = \frac{\arccos\tilde\sigma}{\sqrt{1-\tilde\sigma^2}}=:d_*\,.
$$
Recalling that $\tilde d=\frac{2\theta}{\theta-1}d$ and $d=\frac{\ln2}{b}$ it immediately follows that the previous condition is equivalent to $b>b_*$.
\end{proof}

\begin{remark} \label{rm:b*}
It is important to compare $b_*$ with the critical value $b_0=\frac{2}{\gamma-1}$: by introducing $\rho := 1-2^{-(\gamma-1)}\in(0,1)$, we have in terms of $\rho$
\begin{equation*}
\frac{b_*}{b_0}
= -\frac{\ln(1-\rho)}{\rho} \frac{\sqrt{\rho-\frac{\rho^2}{4}}}{\arccos\bigl(1-\frac{\rho}{2}\bigr)}
=: p(\rho)\,.
\end{equation*}
As $p(\rho)>1$ for all $\rho\in(0,1)$, the inequality $b_*>b_0$ holds for every value of $\gamma>1$. Notice also that in the limit $\gamma\to1^+$ the ratio $\frac{b_*}{b_0}$ converges asymptotically to 1, while in the regime of large homogeneity the two values $b_0$ and $b_*$ are separated: $b_0\to0$ and $b_*\to\frac{3\sqrt{3}\ln2}{\pi}$ as $\gamma\to\infty$.
\end{remark}


\subsection{Convergence to the constant solution} \label{subsect:btoinf}
The result in Proposition~\ref{prop:stability} allows us to show that the solution $\vphi$ to \eqref{equation3} converges to the constant value $\Phi_\infty$, for large values of the free parameter $b$ (that is, for small values of the delay $d$).

\begin{proposition}[Convergence to the constant solution for $b\to\infty$] \label{prop:btoinf}
	There exists $b_2>b_0$ such that for every $b>b_2$ the solution $\vphi$ to \eqref{equation3} is globally defined, positive, and satisfies
	\begin{equation*}
	\vphi(z) \to \Phi_\infty \qquad \text{as }z\to\infty.
	\end{equation*}
	Equivalently, for $b>b_2$ the solution $\Phi$ to \eqref{equation} is defined in $[0,\infty)$, positive, and satisfies $\lim_{x\to\infty}\Phi(x)=\Phi_\infty$.
\end{proposition}

\begin{proof}
The idea of the proof is that for $b$ large enough (or equivalently for small values of the delay $d=\frac{\ln2}{b}$) the solution to \eqref{equation3} is close to the function $\psi$ solving the limit problem
\begin{equation} \label{btoinf1}
\begin{cases}
\psi'(z)=\psi(z)-(2^{\gamma-1}-1)\bigl(\psi(z)\bigr)^2\,,\\
\psi(z)\sim e^z  \qquad\qquad\text{as } z\to-\infty,
\end{cases}
\end{equation}
which is explicitly given by
$$
\psi(z)= \frac{e^z}{1+(2^{\gamma-1}-1)e^z}\,.
$$
As $\psi(z)\to\Phi_\infty$ as $z\to\infty$, the conclusion will follow by the asymptotic stability of $\Phi_\infty$ proved in Proposition~\ref{prop:stability}. We simplify the notation by setting $\theta=2^{\gamma-1}$. We also stress the dependence on the parameter $d$ by writing $\vphi_d$ for the solution to \eqref{equation3}.

\smallskip

By the change of variables \eqref{variables3}, $\vphi_d$ is transformed into the solution $H_d$ to \eqref{equation2} and $\psi$ into the function $\bar{H}(y) = \frac{1}{1+(\theta-1)y}$. In view of Lemma~\ref{lem:wellpos} the functions $H_d$ are locally defined, positive and analytic in a neighborhood of the origin; by monotonicity and positivity one has the lower bound
\begin{equation*}
H_d'(y) \geq -\sigma\bigl(H_d(2^{-\frac{1}{b}}y)\bigr)^2 \geq -\sigma \geq -\theta\,,
\end{equation*}
which implies that $H_d$ remains positive in a uniform interval $[0,y_0]$, with $y_0>0$ independent of $d$. Moreover, by observing that the functions $H_d$ and their derivatives (of any order) are uniformly bounded in $[0,y_0]$, one can pass to the limit in the equation and obtain that $H_d$ converges uniformly to $\bar{H}$ in $[0,y_0]$ as $d\to0^+$.

\smallskip

Going back to the solution $\vphi_d$ to \eqref{equation3} with the change of variables \eqref{variables3}, the previous discussion shows for every value of the parameter $d$ the existence of a unique solution to \eqref{equation3} at least in some interval $(-\infty,z_0]$, where $z_0$ is uniform with respect to $d$, and
\begin{equation} \label{btoinf1b}
\vphi_d\to\psi \qquad\text{uniformly in $(-\infty,z_0]$ as $d\to0^+$.}
\end{equation}
Moreover, by Lemma~\ref{lem:wellpos} $\vphi_d$ can be continued as long as it remains positive, and is uniformly bounded from above by the constant $\frac{1}{\sigma-1}=\frac{1}{\theta2^{-2/b}-1}$; since we are looking at the regime $b\to\infty$, we can hence assume without loss of generality that $\vphi_d$ is bounded from above by a uniform constant $K_1>0$. Therefore, if $z_1$ is the first point at which the solution crosses zero (with $z_1=\infty$ if $\vphi_d$ is everywhere positive), we have
\begin{equation} \label{btoinf2}
	0<\vphi_d(z) \leq K_1, \qquad |\vphi_d'(z)| \leq K_1 + (\theta+1)K_1^2 =: K_2 \qquad\qquad \text{for all } z<z_1\,.
\end{equation}

Define now the function $\eta(z):=\vphi_d(z)-\psi(z)$, where $\psi$ is the solution to \eqref{btoinf1}.
By taking the Taylor expansion of $\vphi_d$ at a point $z$,
\begin{equation*}
	\vphi_d(z-d) = \vphi_d(z) - d\vphi_d'(t_z)\,, \qquad t_z\in(z-d,z),
\end{equation*}
and by the explicit representation of $\psi$ we find that $\eta$ solves
\begin{align*}
	\eta'(z)
	& = \eta(z) - (\theta-1)\bigl(\vphi_d(z)+\psi(z)\bigr)\eta(z) + 2\theta d\vphi_d(z)\vphi_d'(t_z) - \theta d^2\vphi_d'(t_z)^2 \\
	& = \Bigl( 1 - (\theta-1)\bigl(\eta(z)+2\psi(z)\bigr) \Bigr)\eta(z) + 2\theta d\vphi_d(z)\vphi_d'(t_z)-\theta d^2\vphi_d'(t_z)^2 \\
	& = p(z)\eta(z) - (\theta-1)\eta(z)^2 + r(z)\,,
\end{align*}
where we set
\begin{equation*}
p(z) :=\frac{1-(\theta-1)e^z}{1+(\theta-1)e^z}\,,
\qquad
r(z) :=2\theta d\vphi_d(z)\vphi_d'(t_z)-\theta d^2\vphi_d'(t_z)^2\,.
\end{equation*}
Notice that the estimates \eqref{btoinf2} and the explicit definitions of $\psi$ and $p$ give
\begin{equation*}
|p(z)-(\theta-1)\eta(z)| \leq K, \qquad |r(z)| \leq Kd \quad\quad\text{for all }z<z_1,
\end{equation*}
for some uniform constant $K>0$. Therefore for every $z\in[z_0, z_1]$
\begin{equation} \label{btoinf3}
\begin{split}
|\eta(z)| &= \bigg| \eta(z_0)e^{\int_{z_0}^z (p(w)-(\theta-1)\eta(w))\de w} + \int_{z_0}^z e^{\int_{w}^z (p(s)-(\theta-1)\eta(s))\de s}r(w)\de w \bigg| \\
& \leq |\eta(z_0)| e^{K(z-z_0)} + dK e^{K(z-z_0)}\,.
\end{split}
\end{equation}
Recall that by \eqref{btoinf1b} the value $|\eta(z_0)|$ can be made arbitrarily small for $d\to0^+$. Therefore \eqref{btoinf3}, together with the fact that $\psi(z)\to\Phi_\infty$ as $z\to\infty$, implies that for any given $\e>0$ we can find $R>z_0$ and $d_0>0$ such that for all $d<d_0$
\begin{equation*}
\big|\vphi_d(z)-\Phi_\infty\big| \leq |\eta(z)| + |\psi(z)-\Phi_\infty| \leq \e \qquad\text{for every }z\in[R-d,R]\,.
\end{equation*}
Notice that the estimate itself implies the positivity of $\vphi_d$, so that it can be extended up to the point $R$. By choosing $\e$ small enough the asymptotic stability of $\Phi_\infty$ proved in Proposition~\ref{prop:stability} yields the conclusion.
\end{proof}


\subsection{Discussion of the oscillating solutions} \label{subsect:oscill}
Numerical simulations for the delay equation \eqref{equation3}, see Figure~\ref{fig:oscill} and the paper \cite{Ley05}, show the emergence of oscillations in the solution $\vphi$ for values of the free parameter $b<b_*$, that is in the instability regime of the constant solution. The amplitude of such oscillations becomes larger as the free parameter $b$ approaches from above a critical value $\bar b$, for which we expect to have an exponentially decaying solution; for $b<\bar b$ we see instead a change of sign in the solution.

If we go back to the function $H$ solving \eqref{equation2} with the change of variable \eqref{variables3}, we see that the oscillating solutions are transformed in profiles having a stair-like structure, where a sequence of large plateaus are spaced out by transition regions. The overall decay is that of a power law,
\begin{equation*}
H(y) \geq\frac{c}{1+y}\,,
\end{equation*}
with the constant $c>0$ becoming smaller and smaller as the parameter $b$ approaches the critical value $\bar b$, for which we expect an exponentially decaying solution.
\begin{figure}
	\centering
	\includegraphics{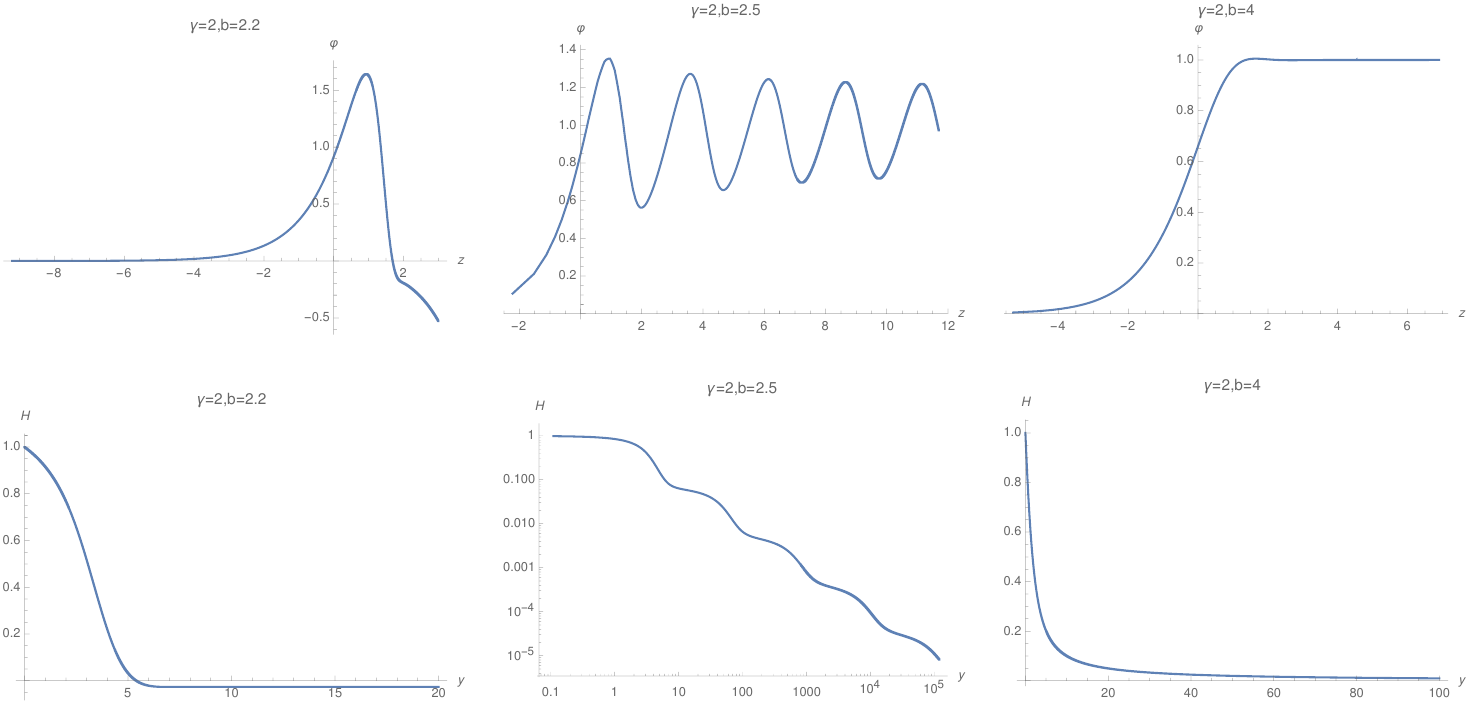}
	\caption{Numerical plots of the solution $\vphi$ to \eqref{equation3} (first row) and of the corresponding solution $H$ to \eqref{equation2} (second row), for different values of the free parameter $b$ and for $\gamma=2$. The two functions $\vphi$ and $H$ are related by the change of variables \eqref{variables3}. For small values of $b$ we have a change of sign of the solution (left), and for large values of $b$ the solution $\vphi$ converges to the constant $\Phi_\infty$, which corresponds to a power law decay $H(y)\sim\frac{1}{y}$ of $H$ (right). For intermediate values $b<b_*$ the solution $\vphi$ develops oscillations, which are reflected in a stair-like structure of the function $H$, plotted in a logaritmic scale (center).}
	\label{fig:oscill}
\end{figure}

We present below an iterative argument which sheds some light on these stair-like solutions observed numerically. Our discussion will be restricted to the limit case in which we send to infinity the value $\gamma$ of the homogeneity. In this case a rescaling by the parameter $\sigma$, namely $h(y):= H(\frac{y}{\sigma})$, leads to the equation
\begin{equation*}
\begin{cases}
h'(y) = -\bigl(h(2^{-\frac1b}y)\bigr)^2 + \frac{1}{\sigma}\bigl(h(y)\bigr)^2\,,\\
h(0)=1.
\end{cases}
\end{equation*}
Since $\sigma\to\infty$ as $\gamma\to\infty$, we directly look at the limit case in which we neglect the second term on the right-hand side of the equation; moreover we write $2^{-\frac1b}=\frac12(1+\e)$, so that the problem can be rewritten as
\begin{equation} \label{oscill1}
\begin{cases}
h'(y) = -\bigl(h\bigl(\frac{y}{2}(1+\e)\bigr)\bigr)^2\,,\\
h(0)=1.
\end{cases}
\end{equation}
The critical value of the free parameter $b$ is in this case $\bar{b}=1$, corresponding to $\e=0$, for which the explicit solution is $\bar{h}(y)=e^{-y}$. We will show in Proposition~\ref{prop:oscill} that solutions for small values of $\e>0$ present the stair-like structure described above, while for $\e<0$ a change of sign takes place.

The advantage of looking at this limit case is that the explicit solution is known for the critical value of the parameter. However, we expect that the same picture holds in the general case: assuming that for some critical value $\bar{b}$ there is an exponentially decaying solution, for values of $b>\bar b$ close to $\bar{b}$ solutions have the same stair-like structure described here, while for $b<\bar b$ solutions change sign.

\begin{proposition} \label{prop:oscill}
	For every $\e>0$ small enough the solution $h$ to \eqref{oscill1} satisfies the bound
	\begin{equation*}
	h(y) \geq \frac{c\e^{1-\alpha}}{1+y}\,,
	\end{equation*}
	where $\alpha\in(0,\frac13)$ is any given number, and $c>0$ is a constant independent of $\e$. For $\e<0$ the solution becomes negative.
\end{proposition}

\begin{proof}
We construct a perturbative solution to \eqref{oscill1} in the form
\begin{equation} \label{oscill2}
h(y) = \bar{h}(y) + \e w_1(y) + \e^2 w_2(y) + r(y)\,,
\end{equation}
where $\bar{h}(y)=e^{-y}$. We will first show that $h$ remains of order $\e$ in a large interval, at least until $y\sim\e^{-\alpha}$. Then, rescaling the solution by $\e$, we obtain a new function $h_1$ solving the same equation as $h$, and taking prescribed initial values of order one in an interval close to the origin; this function has therefore the same behaviour as $h$, that is, it is approximately constant of order $\e$ in a large interval. The whole picture obtained by the iteration of this argument gives the stair-like solution observed numerically, where the same structure with a large plateau is rescaled and repeated.

\medskip\noindent\textit{Step 1: computation of the asymptotics of $h$.}
We take a Taylor expansion
\begin{align*}
\textstyle h\bigl(\frac{y}{2}(1+\e)\bigr)
& \textstyle = \bar{h}\point + \e \Bigl( \frac{y}{2}\bar{h}'\point + w_1\point \Bigr)
+ \e^2 \Bigl( \frac{y^2}{8}\bar{h}''(\xi) + \frac{y}{2}w_1'(\xi_1) + w_2\point \Bigr) \\
& \textstyle \qquad + \e^3\frac{y}{2}w_2'(\xi_2) + r\point + \e\frac{y}{2}r'(\xi_3)
\end{align*}
for suitable points $\xi,\xi_1,\xi_2,\xi_3\in(\frac{y}{2},\frac{y}{2}(1+\e))$. The right-hand side of equation \eqref{oscill1} becomes
\begin{align*}
\textstyle -\bigl(h\bigl(\frac{y}{2}(1+\e)\bigr)\bigr)^2
& \textstyle = - \bigl(\bar{h}\point\bigr)^2 - 2\e\bar{h}\point \Bigl[ \frac{y}{2}\bar{h}'\point + w_1\point \Bigr]
- \e^2 \Bigl[ \frac{y}{2}\bar{h}'\point + w_1\point \Bigr]^2 \\
& \textstyle \qquad - 2\e^2\bar{h}\point \Bigl[ \frac{y^2}{8}\bar{h}''(\xi) + \frac{y}{2}w_1'(\xi_1) + w_2\point \Bigr] - 2\bar{h}\point r\point + \tilde{r}(y)
\end{align*}
where with $\tilde{r}$ we denote the collection of all the remainder terms:
\begin{align} \label{oscill2a}
\tilde{r}(y)
& \textstyle = - \bigl(r\point\bigr)^2 -2\e\Bigl( \frac{y}{2}\bar{h}'\point + w_1\point \Bigr)r\point - \e y\bigl(\bar{h}\point+r\point\bigr) r'(\xi_3) -\e^2\frac{y^2}{4}\bigl(r'(\xi_3)\bigr)^2 \nonumber\\ 
& \textstyle\qquad - \e^2 y \Bigl( \frac{y}{2}\bar{h}'\point + w_1\point \Bigr) r'(\xi_3) - 2\e^2 \Bigl( \frac{y^2}{8}\bar{h}''(\xi) + \frac{y}{2}w_1'(\xi_1) + w_2\point \Bigr) r\point \nonumber\\
& \textstyle\qquad - 2\e^3 \Bigl( \frac{y}{2}\bar{h}'\point + w_1\point \Bigr)\Bigl( \frac{y^2}{8}\bar{h}''(\xi) + \frac{y}{2}w_1'(\xi_1) + w_2\point \Bigr) \nonumber\\
& \textstyle\qquad - \e^3 y\Bigl( \frac{y^2}{8}\bar{h}''(\xi) + \frac{y}{2}w_1'(\xi_1) + w_2\point \Bigr) r'(\xi_3) - \e^3y \Bigl(\bar{h}\point + r\point\Bigr)w_2'(\xi_2) \nonumber\\
& \textstyle\qquad - \e^4\Bigl( \frac{y^2}{8}\bar{h}''(\xi) + \frac{y}{2}w_1'(\xi_1) + w_2\point \Bigr)^2 - \e^4y\Bigl( \frac{y}{2}\bar{h}'\point + w_1\point + \frac{y}{2}r'(\xi_3)\Bigr)w_2'(\xi_2) \nonumber\\
& \textstyle\qquad - \e^5 y \Bigl( \frac{y^2}{8}\bar{h}''(\xi) + \frac{y}{2}w_1'(\xi_1) + w_2\point \Bigr) w_2'(\xi_2) - \e^6\frac{y^2}{4}\bigl(w_2'(\xi_2)\bigr)^2\,.
\end{align}
The functions $w_1$ and $w_2$ solve the following two equations:
\begin{equation} \label{oscill2b}
\begin{split}
w_1'(y) & \textstyle = -2e^{-\frac{y}{2}}w_1\bigl(\frac{y}{2}\bigr) +ye^{-y}\,,\\
w_2'(y) & \textstyle = -2e^{-\frac{y}{2}}w_2\bigl(\frac{y}{2}\bigr) - \bigl(w_1\bigl(\frac{y}{2}\bigr)\bigr)^2 - \frac{y^2}{4}e^{-y} - \frac{y^2}{4}e^{-\frac{y}{2}}e^{-\xi} - ye^{-\frac{y}{2}}w_1'(\xi_1) + ye^{-\frac{y}{2}}w_1\bigl(\frac{y}{2}\bigr)\,,
\end{split}
\end{equation}
with $w_1(0)=0$, $w_2(0)=0$; the remainder $r$ solves instead
\begin{align} \label{oscill2c}
r'(y) = \textstyle -2e^{-\frac{y}{2}} r\bigl(\frac{y}{2}\bigr) + \tilde{r}(y)\,,\qquad r(0)=0.
\end{align}
All these equations can be seen as versions of the same linearized problem, with different source terms; by rescaling with an exponential factor we can reduce it to the linear delay equation discussed in Section~\ref{sect:linear}. Indeed by setting $\tilde{w}_1(y) = e^yw_1(y)$ we have
\begin{equation*}
\tilde{w}_1'(y) = \tilde{w}_1(y) - 2 \tilde{w}_1\bigl(\textstyle\frac{y}{2}\bigr) + y, \qquad \tilde{w}_1(0)=0,
\end{equation*}
and this equation can be solved using the fundamental solution computed in Lemma~\ref{lem:linearproblem}: with the notation introduced in that lemma,
\begin{equation} \label{oscill3}
w_1(y) = e^{-y}\tilde{w}_1(y) = e^{-y} \int_0^y G(y,\eta)\eta\de\eta = \int_0^y Q(\eta)\eta\de\eta + e^{-y}\int_0^y \widetilde{G}(y,\eta)\eta\de\eta\,.
\end{equation}
As the functions $Q$ and $\widetilde{G}$ obey the estimates \eqref{linearbounds1}--\eqref{linearbounds2}, from \eqref{oscill3} we obtain for a fixed $\beta\in(0,\frac12)$ the asymptotics
\begin{equation}\label{oscill4}
w_1(y) = \int_0^\infty Q(\eta)\eta\de\eta + O(e^{-\beta y}) \qquad\text{as }y\to\infty.
\end{equation}
Notice that the integral above is a strictly positive quantity: indeed, the function $Q$ is defined in \eqref{linearG} as an alternating series, $Q(\eta)=\sum_{n=0}^\infty (-1)^na_n e^{-2^n\eta}$, $a_n>0$, and it is easily seen that the sequence of the coefficients $(a_n)_n$ is strictly decreasing starting from $n=2$. Therefore it is sufficient to compute the contribution to the integral in \eqref{oscill4} by the first two terms of the sum, as the rest gives a strictly positive quantity:
\begin{equation} \label{oscill5}
\begin{split}
c_0 := \int_0^\infty Q(\eta)\eta\de \eta
& = \int_0^\infty e^{-\eta}\eta\de\eta -4\int_0^\infty e^{-2\eta}\eta\de\eta +\int_0^\infty\sum_{n=2}^\infty (-1)^na_n e^{-2^n\eta}\eta\de\eta \\
& = \int_0^\infty\sum_{n=2}^\infty (-1)^na_n e^{-2^n\eta}\eta\de\eta >0\,.
\end{split}
\end{equation}

We now look at the equation \eqref{oscill2b} for $w_2$: this is the same type of linearized problem as before; therefore denoting by $s(y)$ the source term on the right-hand side of the equation we have the representation formula for the solution
\begin{align*}
w_2(y) = e^{-y}\int_0^y G(y,\eta)e^{\eta}s(\eta)\de\eta = \int_0^y Q(\eta)e^\eta s(\eta) \de\eta + e^{-y}\int_0^y \widetilde{G}(y,\eta)e^\eta s(\eta)\de\eta\,.
\end{align*}
Notice that all the source terms have the exponential decay $O(ye^{-\frac{y}{2}})$ as $y\to\infty$, except for the term $- \bigl(w_1\bigl(\frac{y}{2}\bigr)\bigr)^2 \sim -c_0^2 + O(e^{-\frac{\beta y}{2}})$. Therefore using the definition for $Q$ and the decay estimates \eqref{linearbounds1}--\eqref{linearbounds2} we have as $y\to\infty$
\begin{equation} \label{oscill6}
\begin{split}
w_2(y) &= \int_0^y s(\eta)\de\eta + \sum_{n=1}^\infty (-1)^n a_n \int_0^y e^{(1-2^n)\eta}s(\eta)\de\eta + e^{-y}\int_0^y \widetilde{G}(y,\eta)e^\eta s(\eta)\de\eta \\
&= -c_0^2 y + k_0 + O(e^{-\frac{\beta y}{2}}) \qquad\text{as }y\to\infty
\end{split}
\end{equation}
for some constant $k_0$ depending on all the source terms.

Having the two asymptotics \eqref{oscill4} and \eqref{oscill6} for $w_1$ and $w_2$, we can compute the next order correction in the expansion \eqref{oscill2}, which is given by the function $r$ solving \eqref{oscill2c}. Using again the fundamental solution $G$ of the linear problem and the bound $|G(y,\eta)|\leq C_0e^{y-\eta}$, which follows from \eqref{linearbounds1}--\eqref{linearbounds2}, we have
\begin{align} \label{oscill7c}
|r(y)| = \bigg| e^{-y}\int_0^y G(y,\eta) e^\eta \tilde{r}(\eta)\de\eta\bigg| \leq C_0 \displaystyle\int_0^y |\tilde{r}(\eta)| \de\eta\,.
\end{align}
Let $\alpha\in(0,\frac13)$ be any fixed number. We claim that for some constant $C>0$
\begin{equation} \label{oscill7}
|r(y)| \leq C\e^{3-3\alpha} \qquad\text{for all }y\leq\frac{1}{\e^\alpha}.
\end{equation}
This estimate can be proved by means of a continuation argument: indeed, assume that
\begin{equation} \label{oscill7b}
|r(\eta)| + |r'(\eta)| \leq C\e^{3-3\alpha} \qquad\text{for all }\eta\leq y,
\end{equation}
for some $y\in(0,\e^{-\alpha})$. This estimate is certainly true for $y$ small enough, as $r(0)=r'(0)=0$. By plugging \eqref{oscill7b} into the explicit expression \eqref{oscill2a} of $\tilde{r}$ one finds after straightforward estimates, using also the asymptotics \eqref{oscill4} and \eqref{oscill6} for $w_1$ and $w_2$, that for all $\eta\leq y$
\begin{align*}
|\tilde{r}(\eta)|
& \textstyle \leq C_1\e^{2\alpha}\e^{3-3\alpha}
\end{align*}
for a new constant $C_1$ depending ultimately on $C$, $c_0$, and $k_0$. In turn, by \eqref{oscill7c} we obtain
\begin{equation*}
|r(\eta)| \leq C_0 C_1 \e^{2\alpha} \e^{3-3\alpha}\eta \leq C_0C_1\e^\alpha\e^{3-3\alpha}\,,
\end{equation*}
while \eqref{oscill2c} yields
\begin{equation*}
|r'(y)| \leq 2e^{-\frac{y}{2}}\big|r\textstyle\point\big| + |\tilde{r}(y)|
\leq \bigl(2C_0C_1 + C_1\bigr)\e^\alpha \e^{3-3\alpha}\,.
\end{equation*}
Therefore, if $\e$ is small enough, from the previous two estimates it follows
\begin{equation*}
|r(y)| + |r'(y)| \leq \frac{1}{2}C\e^{3-3\alpha}\,.
\end{equation*}
Since we obtained this inequality just assuming \eqref{oscill7b} for some $y\in(0,\e^{-\alpha})$, we conclude by a continuation argument that the claim \eqref{oscill7} holds for every $y\leq\e^{-\alpha}$.

Finally, by collecting \eqref{oscill4}, \eqref{oscill6} and \eqref{oscill7} we obtain from \eqref{oscill2}
\begin{equation*}
h(y) = c_0\e - c_0^2\e^2 y + k_0\e^2 + O(e^{-y} + \e e^{-\beta y} + \e^2e^{-\frac{\beta y}{2}} + \e^{3-3\alpha}).
\end{equation*}
This asymptotics is valid at least until $y\leq \e^{-\alpha}$. In particular in the region $y\sim \e^{-\alpha}$ the exponential terms give a very small contribution and can be neglected. Therefore we can write
\begin{equation} \label{oscill8}
h(y) = c_0\e + k_0\e^2 - c_0^2\e^2 y + O(\e^{3-3\alpha}) \qquad\text{for }y\in\Bigl[\frac{\e^{-\alpha}}{c_0+k_0\e},\frac{2\e^{-\alpha}}{c_0+k_0\e}\Bigr].
\end{equation}
Notice that if $\e$ is negative a change of sign in the solution takes place.

\medskip\noindent
\textit{Step 2: rescaling.} The idea is now to rescale the function $h$ by $\e$, and simultaneously to rescale the independent variable by $\e$, in order to obtain a new function solving an equation with the same structure as before and taking prescribed values of order one in a small interval close to the origin; the asymptotic of this function will be very similar to the one computed for $h$ in the previous step, so that we will repeat the argument and iterate the method. Recall that, given a solution $h$ to \eqref{oscill1} (without initial condition), we obtain a new solution by the rescaling $\tilde{h}(y)=\lambda h(\lambda y)$. Then we define
\begin{equation*}
h_1(y) = \frac{1}{c_0\e+k_0\e^2} h\Bigl(\frac{y}{c_0\e+k_0\e^2}\Bigr)\,,
\end{equation*}
and in view of \eqref{oscill8} we have
\begin{equation*}
h_1(y) = 1-y+O(\e^{2-3\alpha}) \qquad\text{for } y\in[\e^{1-\alpha},2\e^{1-\alpha}]\,.
\end{equation*}
We therefore consider the initial value problem
\begin{equation} \label{oscill8b}
\begin{cases}
h_1'(y) = -\bigl(h_1\bigl(\frac{y}{2}(1+\e)\bigr)\bigr)^2 & y>2\e^{1-\alpha},\\
h_1(y)= 1 - y + O(\e^{2-3\alpha}) & y\in[\e^{1-\alpha},2\e^{1-\alpha}].
\end{cases}
\end{equation}
This problem is very similar to the original one, except for the fact that we prescribe here the values of $h_1$ in an interval close to the origin instead that at the single point $y=0$. We write as in \eqref{oscill2}
\begin{equation}  \label{oscill9}
h_1(y) = \bar{h}(y) + \e w_{1,1}(y) + \e^2 w_{1,2}(y) + r_1(y)\,,
\end{equation}
where $w_{1,1}$, $w_{1,2}$ and $r_1$ solve the same equations as $w_1$, $w_2$ and $r$ in the previous step, namely \eqref{oscill2b}--\eqref{oscill2c}, for $y>2\e^{1-\alpha}$ and take prescribed values
\begin{equation} \label{oscill9b}
w_{1,1}(y)= g(y), \qquad
w_{1,2}(y)= r_1(y)= 0 \qquad
\text{for } y\in[\e^{1-\alpha},2\e^{1-\alpha}]
\end{equation}
with $g(y)=O(\e^{1-3\alpha})$.

The asymptotics of $h_1$ can be now computed in the same manner as before; here we only sketch the argument without giving the full details. The function $w_{1,1}$ can be written in terms of the fundamental solution of the linear problem using the representation formula
\begin{align*}
w_{1,1}(y)
& = e^{-y} \int_{2\e^{1-\alpha}}^y G(y,\eta)\eta\de\eta + e^{-y}G(y,2\e^{1-\alpha})e^{2\e^{1-\alpha}}g(2\e^{1-\alpha}) \\
& \qquad - 2e^{-y} \int_{2\e^{1-\alpha}}^{\min\{y,4\e^{1-\alpha} \}} G(y,\eta)e^{\frac{\eta}{2}}g\Bigl(\frac{\eta}{2}\Bigr)\de\eta\,.
\end{align*}
The asymptotics of the first term is, as in the previous step,
\begin{align*}
e^{-y} \int_{2\e^{1-\alpha}}^y G(y,\eta)\eta\de\eta
= w_1(y) - e^{-y}\int_0^{2\e^{1-\alpha}} G(y,\eta)\eta\de\eta
= c_0 + O(e^{-\beta y }) + O(\e^{2-2\alpha})\,.
\end{align*}
For the other two terms in the expression of $w_{1,1}$ one finds, using the formula \eqref{linearG} for the fundamental solution and the bounds \eqref{linearbounds1}--\eqref{linearbounds2},
\begin{align*}
e^{-y}G(y,2\e^{1-\alpha})e^{2\e^{1-\alpha}}g(2\e^{1-\alpha}) - 2e^{-y} \int_{2\e^{1-\alpha}}^{\min\{y,4\e^{1-\alpha} \}} & G(y,\eta)e^{\frac{\eta}{2}}g\Bigl(\frac{\eta}{2}\Bigr)\de\eta \\
& = Q(2\e^{1-\alpha})e^{2\e^{1-\alpha}}g(2\e^{1-\alpha}) + O(e^{-\beta y})\,.
\end{align*}
Notice that the first term on the right-hand side is just a constant of order $\e^{1-3\alpha}$ (that is the order of the initial datum $g$). Therefore we obtain the following asymptotics for the function $w_{1,1}$:
\begin{equation} \label{oscill10a}
w_{1,1}(y) = c_1(\e) + O(\e^{2-2\alpha} + e^{-\beta y})
\end{equation}
for a positive constant $c_1(\e)=c_0+Q(2\e^{1-\alpha})e^{2\e^{1-\alpha}}g(2\e^{1-\alpha})$, depending on $\e$, which differs from $c_0$ only up to a small remainder of order $\e^{1-3\alpha}$.

We now look at the term $w_{1,2}$, which solves the same equation \eqref{oscill2b} as $w_2$ (with $w_1$ replaced by $w_{1,1}$) for $y>2\e^{1-\alpha}$, with homogeneous initial condition \eqref{oscill9b}. Denote by $s_1(y)$ the sources in the equation for $w_{1,2}$. As in the previous step, all the terms in $s_1$ decay exponentially, except for $-(w_{1,1}(\frac{y}{2}))^2 \sim - c_1(\e)^2 + O(\e^{2-2\alpha}+e^{-\frac{\beta y}{2}})$. Therefore using the representation formula we find
\begin{equation} \label{oscill10b}
w_{1,2}(y) = e^{-y} \int_{2\e^{1-\alpha}}^y G(y,\eta)e^\eta s_1(\eta)\de\eta = -c_1(\e)^2y + k_1 + O(\e^{1-\alpha}) + O(\e^{2-2\alpha}y) + O(e^{-\frac{\beta y}{2}})\,,
\end{equation}
where this equality can be deduced following the same passages leading to \eqref{oscill6}; here $k_1$ is a constant depending on all the source terms. 

Finally, using the aymptotics of $w_{1,1}$ and $w_{1,2}$ just computed, one can apply the same continuation argument as in the previous step to obtain
\begin{equation} \label{oscill10c}
|r_1(y)| \leq C\e^{3-3\alpha} \qquad\text{for all }y\leq\frac{1}{\e^\alpha}.
\end{equation}

By collecting all the asymptotic behaviours \eqref{oscill10a}, \eqref{oscill10b}, \eqref{oscill10c} we deduce from \eqref{oscill9}
\begin{equation*}
h_1(y) = c_1(\e)\e + k_1\e^2 - c_1(\e)^2\e^2y + O(\e^{3-3\alpha} + e^{-\frac{\beta y}{2}}),
\end{equation*}
valid for $y\leq\e^{-\alpha}$. For $y\sim \e^{-\alpha}$, the exponential terms are negligible and we finally obtain an asymptotics which resembles \eqref{oscill8}:
\begin{equation} \label{oscill11}
h_1(y) = c_1(\e)\e + k_1\e^2 - c_1(\e)^2\e^2y + O(\e^{3-3\alpha}) \qquad \text{for }y\in\biggl[\frac{\e^{-\alpha}}{c_1(\e)+k_1\e},\frac{2\e^{-\alpha}}{c_1(\e)+k_1\e}\biggr].
\end{equation}

\medskip\noindent
\textit{Step 3: iteration.}
We can now iterate the previous rescaling argument: at the $n$-th step the functions $h_1,\ldots,h_{n}$ have been defined, with the asymptotics
\begin{equation} \label{oscill12}
h_{j}(y) = c_j(\e)\e + k_j\e^2 - c_j(\e)^2\e^2y + O(\e^{3-3\alpha}) \qquad \text{for }y\in\biggl[\frac{\e^{-\alpha}}{c_j(\e)+k_j\e},\frac{2\e^{-\alpha}}{c_j(\e)+k_j\e}\biggr].
\end{equation}
Then we set
\begin{equation*}
h_{n+1}(y) = \frac{1}{c_n(\e)\e+k_n\e^2} h_n\Bigl(\frac{y}{c_n(\e)\e+k_n\e^2}\Bigr)\,,
\end{equation*}
so that the function $h_{n+1}$ also solves the same equation as $h$ for $y\geq 2\e^{1-\alpha}$ and in view of \eqref{oscill12} takes the initial values
\begin{equation} \label{oscill13}
h_{n+1}(y) = 1-y+O(\e^{2-3\alpha}) \qquad\text{for } y\in[\e^{1-\alpha},2\e^{1-\alpha}]\,.
\end{equation}
This is the very same problem as \eqref{oscill8b}, so that the previous step yields that the asymptotics \eqref{oscill12} is valid also for $j=n+1$. Since the structure of the problem remains the same in all steps, one can check that the constants $c_n(\e)$, $k_n$ remain of order one for every $n$.

Therefore we obtain the formula for $h$
\begin{equation}
h(y) = p_{n}h_{n+1}(p_{n}y) \qquad \text{where } p_{n} := \prod_{j=0}^n \bigl(c_j(\e)\e+k_j\e^2\bigr).
\end{equation}
This representation explains the stair-like structure observed numerically. For every fixed $\theta\in[1,2]$ and for $y=\frac{\theta\e^{1-\alpha}}{p_n}$ one then gets in view of \eqref{oscill13}
\begin{align*}
h(y) = p_n h_{n+1}(\theta\e^{1-\alpha}) \simeq p_n = \frac{\theta\e^{1-\alpha}}{y}\,,
\end{align*}
which proves the estimate in the statement.
\end{proof}


\section{Formal asymptotics in the regime \texorpdfstring{\boldmath$\gamma\to1$}{of homogeneity close to 1}} \label{sect:gammato1}

In this section we discuss the features of solutions to \eqref{equation} for values of the homogeneity $\gamma$ close to 1. We work with the variables introduced in \eqref{variables2}.  The regime $\gamma\to 1^+$ corresponds to $\sigma\to1^+$ and $b\to\infty$, therefore it is convenient to introduce two small parameters $\e>0$, $\delta>0$ by setting
\begin{equation*}
\e = 1-2^{-\frac1b}, \qquad \delta = 2^{\gamma-1-\frac2b}-1
\end{equation*}
and to write the equation \eqref{equation2} in the following form:
\begin{equation} \label{gammato11}
	\begin{cases}
		H'(y) = -(1+\delta)\bigl(H(y(1-\e))\bigr)^2 + \bigl(H(y)\bigr)^2\,,\\
		H(0)=1\,.
	\end{cases}
\end{equation}


\subsection{The limit case\texorpdfstring{: \boldmath$\gamma=1$}{}} \label{subsect:gamma1}

We recall what is known in the case of the diagonal kernel $K(x,y)=x^2\delta(x-y)$. Most of the following observations have already been made in \cite{Ley05}, but a wrong conclusion was drawn in \cite[Theorem~2]{Ley05}, which we correct here.

In general, in the case of homogeneity $\gamma=1$ one looks for self-similar solutions to Smoluchowski's equation \eqref{eq:smol} in the form $f(\xi,t)=e^{-2bt}F(\xi e^{-bt})$, $x=\xi e^{-bt}$, $b>0$, for which the self-similar profile $F$ solves
\begin{equation*}
- b\bigl(2F(x)+x F'(x)\bigr) = \frac12\int_0^x K(x-y,y)F(x-y)F(y)\de y - \int_0^\infty K(x,y)F(x)F(y)\de y\,.
\end{equation*}
In the case of the diagonal kernel, the equation takes the form
\begin{equation*}
-2b F(x) - bx F'(x) = \frac14\Bigl(\frac{x}{2}\Bigr)^2 \bigl(F(x/2)\bigr)^2 - x^2 \bigl( F(x) \bigr)^2\,,
\end{equation*}
and we look for a solution with $F(x)\sim c_0x^{-2}$ as $x\to 0^+$. With the rescaling $\Phi(x)=x^2 F(x)$, which removes the singularity at the origin, the equation becomes
\begin{equation} \label{gamma1a}
bx\Phi'(x) = - \bigl(\Phi(x/2)\bigr)^2 + \bigl(\Phi(x)\bigr)^2\,, \qquad \Phi(0)=1,
\end{equation}
where we have normalized $c_0=1$ by a simple rescaling.

We wish to find a unique value of the parameter $b$ for which \eqref{gamma1a} has a positive solution decaying exponentially to zero. This shooting parameter can be interpreted as the exponent $b$ yielding the expansion of the clusters as $e^{bt}$. By integrating \eqref{gamma1a} between $x_0$ and $x$ and sending $x_0\to0^+$, using the normalization $\Phi(0)=1$, we can write the equation in the integrated form
\begin{equation} \label{gamma1b}
b\Phi(x) = \int_{\frac{x}{2}}^x \frac{1}{s}\bigl(\Phi(s)\bigr)^2\de s + b - \ln2\,,
\end{equation}
and we see that the critical value $b$ for which the solution decays to zero is the unique value such that the additive constant in \eqref{gamma1b} is zero, that is
\begin{equation} \label{gamma1c}
b=\ln2\,.
\end{equation}
Notice that a $C^1$ solution to \eqref{gamma1a} for $b=\ln2$ necessarily satisfies $\Phi'(0)=0$, see Remark~\ref{rm:gamma1} below; in this case the only analytic solution at the origin is the constant solution $\Phi\equiv1$, as one can directly check by looking for a solution in power series. However, we can construct a different solution in the form
\begin{equation} \label{gamma1d}
\Phi(x) = 1 + \sum_{n=1}^\infty a_n x^{n\alpha},
\end{equation}
where $\alpha$ is the unique positive solution to
\begin{equation} \label{gamma1e}
\frac{b\alpha}{2(1-2^{-\alpha})}=1
\end{equation}
(notice in particular that $\alpha>2$). By using the ansatz \eqref{gamma1d} in the equation \eqref{gamma1a}, it is straightforward to determine the coefficients $a_n$ by the recurrence relation
\begin{equation*}
\Bigl( \frac{nb\alpha}{1-2^{-n\alpha}} -2 \Bigr) a_n = \sum_{m=1}^{n-1}a_ma_{n-m}\,, \qquad n\geq 2,
\end{equation*}
where the value of $a_1$ is arbitrary. The series \eqref{gamma1d} is locally convergent in a finite interval around zero, and from the integrated equation \eqref{gamma1b} we see that $\Phi$ is always positive. Choosing $a_1<0$, the solution is monotone decreasing in a neighborhood of the origin, and the equation \eqref{gamma1a} implies that it remains decreasing also for larger $x>0$. Furthermore, since $\Phi$ is uniformly bounded, by standard existence and uniqueness results for ordinary differential equations it can be uniquely continued to all the positive real axis. Finally using \eqref{gamma1b} and the monotonicity of $\Phi$ we find
$$
(\ln2)\Phi(x) \leq \bigl( \Phi(x/2)\bigr)^2 \int_{\frac{x}{2}}^x\frac{1}{s}\de s = (\ln2)\bigl( \Phi(x/2)\bigr)^2\,,
$$
from which the exponential decay to zero can be deduced. We summarize the previous discussion in a proposition.

\begin{proposition}[The case $\gamma=1$] \label{prop:gamma1}
	Let $b=\ln2$. For every $a_1<0$ the equation \eqref{gamma1a} has a unique solution in the form \eqref{gamma1d}, where $\alpha$ is the unique positive solution to \eqref{gamma1e}. Such solution is positive, monotonically decreasing, and converges to zero exponentially as $x\to\infty$.
\end{proposition}

\begin{remark} \label{rm:gamma1}
If $\Phi$ is a $C^1$ solution to \eqref{gamma1a}, then the right-hand side of the equation obeys
\begin{align*}
\lim_{x\to0^+}\frac{(\Phi(x))^2 - (\Phi(x/2))^2}{x} = \Phi'(0)\,,
\end{align*}	
which forces $(b-1)\Phi'(0)=0$. Therefore for a generic value of $b\neq1$, the differentiability at the origin of a solution to \eqref{gamma1a} implies that $\Phi'(0)=0$.

In the case $b=1$, the argument given by Leyvraz in \cite[Theorem~2]{Ley05} shows that one can construct an analytic solution in a neighborhood of the origin for every (negative) value of $\Phi'(0)$. But in this case the integrated equation is
\begin{equation*}
\Phi(x) = \int_{\frac{x}{2}}^x \frac{1}{s}\bigl(\Phi(s)\bigr)^2\de s + 1-\ln2
\end{equation*}
from which we see that this solution cannot decay exponentially to zero -- here is where Leyvraz's argument is wrong: indeed in formula (5.17) in \cite{Ley05} there is a missing additive constant. Therefore the solution found by Leyvraz is an analytic, nonconstant solution to \eqref{gamma1a} for $b=1$; such solution is positive and monotone decreasing, but it does not decay to zero as claimed in \cite{Ley05}, and it actually converges to the constant $\frac{1-\ln2}{\ln2}$.
\end{remark}


\subsection{Formal asymptotics} \label{subsect:gammato1}

We now go back to the formal analysis of the equation \eqref{gammato11} in the asymptotic regime $\e\to0^+$, $\delta\to0^+$. We approximate $H(y)\simeq 1-\delta\Psi(y)$, where $\Psi$ solves the linearized equation
\begin{equation} \label{gammato13}
\begin{cases}
\Psi'(y) = 2\Psi(y) - 2\Psi((1-\e)y) +1\,,\\
\Psi(0)=0\,.
\end{cases}
\end{equation}
By looking for a solution in power series $\Psi(y)=\sum_{n=0}^\infty a_ny^n$, we have $a_0=0$, $a_1=1$, and the recurrence relation $(n+1)a_{n+1}=2a_n(1-(1-\e)^n)$ for $n=1,2,\ldots$, which gives
\begin{equation*}
\Psi(y) = y + \sum_{n=1}^\infty \frac{2^n \prod_{k=1}^n[1-(1-\e)^k]}{(n+1)!} y^{n+1}\,.
\end{equation*}

We want to compute the asymptotics of $\Psi$ for large values (of order $\frac{1}{\e}$). To this aim we introduce the variable $\eta=\e y$ and we write the previous series in the form
\begin{equation} \label{gammato14}
\Psi\Bigl(\frac{\eta}{\e}\Bigr) = \frac12\sum_{n=0}^\infty e^{S_n(\eta,\e)}\,,
\quad S_n(\eta,\e) := (n+1)\ln\Bigl(\frac{2\eta}{\e}\Bigr) + \sum_{k=1}^n \ln\bigl(1-(1-\e)^k\bigr) - \sum_{k=1}^{n+1}\ln k \,.
\end{equation}
By elementary manipulations the second term in $S_n$ can be written as
\begin{align*}
\sum_{k=1}^n \ln\bigl(1-(1-\e)^k\bigr)
& = \sum_{k=1}^n \ln(k\e) + \sum_{k=1}^n \ln\Bigl( \frac{1-e^{k\ln(1-\e)}}{k\e} \Bigr) \\
& = \sum_{k=1}^n \ln k + n\ln\e + \sum_{k=1}^n \ln\Bigl( \frac{1-e^{-k\e}}{k\e} \Bigr) + \sum_{k=1}^n \ln \Bigl(\frac{1-e^{k\ln(1-\e)}}{1-e^{-k\e}} \Bigr)\,.
\end{align*}
Therefore by inserting this expression in \eqref{gammato14} we have
\begin{equation} \label{gammato15}
\begin{split}
S_n(\eta,\e) &= (n+1)\ln(2\eta) - \ln\e - \ln(n+1) \\
& \qquad + \sum_{k=1}^n \ln\Bigl( \frac{1-e^{-k\e}}{k\e} \Bigr) + \sum_{k=1}^n \ln \Bigl(\frac{1-e^{k\ln(1-\e)}}{1-e^{-k\e}} \Bigr)\,.
\end{split}
\end{equation}
We further approximate the first sum on the right-hand side of \eqref{gammato15} by an integral, using Euler- Maclaurin formula: setting $f(x)=\ln\bigl(\frac{1-e^{-\e x}}{\e x}\bigr)$, we have
\begin{align*}
\sum_{k=1}^n \ln\Bigl( \frac{1-e^{-k\e}}{k\e} \Bigr)
& = \int_0^n f(x)\de x + \frac12\bigl(f(n)-f(0)\bigr) + \int_0^n B_1(x-[x])f'(x)\de x \\
& = \int_0^n \ln\Bigl( \frac{1-e^{-\e x}}{\e x} \Bigr)\de x + \frac12\ln\Bigl( \frac{1-e^{-n\e}}{n\e} \Bigr) + \int_0^n B_1(x-[x])f'(x)\de x
\end{align*}
(here $B_1(x)=x-\frac12$ is the first Bernoulli polynomial, and $[x]$ denotes the integer part of $x$). Substituting this expression in \eqref{gammato15} we end up with
\begin{equation} \label{gammato16}
\begin{split}
S_n(\eta,\e)
& = (n+1)\ln(2\eta) + \frac{1}{\e}\int_0^{n\e} \ln\Bigl( \frac{1-e^{-t}}{t} \Bigr)\de t + \frac12\ln(1-e^{-n\e}) \\
& \qquad  - \frac32\ln\e - \ln(\sqrt{n}(n+1)) + R_n(\e)\,,
\end{split}
\end{equation}
where
\begin{equation} \label{gammato16b}
R_n(\e) := \sum_{k=1}^n \ln \Bigl(\frac{1-e^{k\ln(1-\e)}}{1-e^{-k\e}} \Bigr) + \int_0^{n\e} \bigl({\textstyle\frac{t}{\e}-[\frac{t}{\e}]-\frac12}\bigr) \frac{e^{-t} - 1 + t e^{-t}}{t(1-e^{-t})} \de t \,.
\end{equation}
By inserting \eqref{gammato16} into \eqref{gammato14} we obtain the following representation formula for $\Psi$:
\begin{equation} \label{gammato17}
\Psi\Bigl(\frac{\eta}{\e}\Bigr) = \frac{1}{2\e^{\frac32}}\sum_{n=0}^\infty \frac{\sqrt{1-e^{-n\e}}}{\sqrt{n}(n+1)} e^{\bar{S}_n(\eta,\e)} e^{R_n(\e)} \,,
\end{equation}
where we denoted by $\bar{S}_n(\eta,\e)$ the leading order part in \eqref{gammato16}, given by the first two terms:
\begin{equation*}
\bar{S}_n(\eta,\e) := (n+1)\ln(2\eta) + \frac{1}{\e}\int_0^{n\e} \ln\Bigl( \frac{1-e^{-t}}{t} \Bigr)\de t \,.
\end{equation*}

We now use Laplace's method (see for instance \cite{BenOrs99}) to find the leading behaviour of the sum \eqref{gammato17} as $\e\to0^+$. We have to identify the largest term in the series, and for this we compute the point $n_*$ at which the expression $\bar{S}_n(\eta,\e)$ is maximal (the other terms will result of lower order). We easily compute the derivative
\begin{equation*}
\frac{\partial\bar{S}_n(\eta,\e)}{\partial n} = \ln(2\eta) + \ln\Bigl(\frac{1-e^{-n\e}}{n\e}\Bigr) \,.
\end{equation*}
For $\eta>\frac12$, the equation $\frac{t}{1-e^{-t}}=2\eta$ has a unique positive root $t_*(\eta)$ and the maximum of the expression $\bar{S}_n(\eta,\e)$ is attained for $n_*=\frac{t_*(\eta)}{\e}$. For $n-n_*$ small we can approximate $\bar{S}_n$ by its Taylor expansion around the point $n_*$,
\begin{align} \label{gammato18}
\bar{S}_n(\eta,\e)
& \simeq \bar{S}_{n_*}(\eta,\e) + \frac12\frac{\partial^2\bar{S}_n(\eta,\e)}{\partial n^2}\bigg|_{n=n_*} \bigl( n -  n_* \bigr)^2 \nonumber \\
& = \ln(2\eta) + \frac{1}{\e}\biggl[ t_*\ln(2\eta) + \int_0^{t_*} \ln\bigl( \textstyle\frac{1-e^{-t}}{t}\bigr)\de t \biggr]
+ \displaystyle \frac{\e}{2t_*} \biggl( \frac{t_*e^{-t_*}}{1-e^{-t_*}}-1 \biggr) \bigl( n - n_* \bigr)^2 \nonumber \\
& = \ln(2\eta) + \frac{W(\eta)}{\e} - \e D(\eta) \bigl( n - \textstyle \frac{t_*(\eta)}{\e}\bigr)^2
\end{align}
where $W(\eta):=\int_0^{t_*(\eta)}\ln\bigl(\frac{2\eta(1-e^{-t})}{t}\bigr)\de t$ and $D(\eta):= - \frac{1}{2t_*} ( \frac{t_*e^{-t_*}}{1-e^{-t_*}}-1 )$ (notice that $W(\eta)$ and $D(\eta)$ are positive quantities by the definition of the point $t_*(\eta)$).

Using the principles of Laplace's method, to compute the asymptotics of the series \eqref{gammato17} as $\e\to0^+$ we can keep only the terms with $n\in[n_*(1-\sigma),n_*(1+\sigma)]$, where $\sigma>0$ is small. The errors that we make with this approximation are indeed subdominant (exponentially small) with respect to the leading order behaviour, which is $e^{\frac{W(\eta)}{\e}}$. To see this, first observe that for values of $n$ outside this interval the value of the leading part $\bar{S}_n$ differs from its maximum by a quantity of order $\frac{1}{\e}$, that is
\begin{equation} \label{gammato18a}
\bar{S}_n(\eta,\e) \leq \frac{W(\eta)}{\e}-\frac{\alpha}{\e} \qquad\text{for } |n-n_*|>\sigma n_*
\end{equation}
for some $\alpha>0$. Fix now a sufficiently large constant $N$ and consider first the terms with $n\leq\frac{N}{\e}$: for such terms one can check from \eqref{gammato16b} by elementary estimates that
\begin{equation} \label{gammato18b}
R_n(\e)\leq 2N \qquad\text{for } n\leq\frac{N}{\e}\,;
\end{equation}
therefore for such $n$'s in the series \eqref{gammato17} we have
\begin{align*}
\sum_{\stackrel{n\leq\frac{N}{\e}}{|n-n_*|>\sigma n_*}} \frac{\sqrt{1-e^{-n\e}}}{2\e^{\frac32}\sqrt{n}(n+1)} e^{\bar{S}_n(\eta,\e)} e^{R_n(\e)} 
\leq N\e^{-\frac53} e^{\frac{W(\eta)-\alpha}{\e}}e^{2N}
\end{align*}
and these terms do not contribute to the asymptotics of \eqref{gammato17} as $\e\to0$. Also the terms in the series with $n>\frac{N}{\e}$ can be neglected: indeed we can write by \eqref{gammato14}
\begin{align*}
S_n(\eta,\e)
&= S_{[\frac{N}{\e}]}(\eta,\e) + (n-[{\textstyle\frac{N}{\e}}])\ln\Bigl(\frac{2\eta}{\e}\Bigr) + \sum_{k=[\frac{N}{\e}]+1}^n \ln\bigl(1-(1-\e)^k\bigr) - \sum_{k=[\frac{N}{\e}]+2}^{n+1}\ln k \\
&\leq S_{[\frac{N}{\e}]}(\eta,\e) + \ln\Bigl(\frac{2\eta}{\e}\Bigr) + (n-[{\textstyle\frac{N}{\e}}]-1) \Bigl( \ln\Bigl(\frac{2\eta}{\e}\Bigr)-\ln\Bigl(\frac{N}{\e}\Bigr) \Bigr) \\
& = S_{[\frac{N}{\e}]}(\eta,\e) + \ln\Bigl(\frac{2\eta}{\e}\Bigr) - (n-[{\textstyle\frac{N}{\e}}]-1) \ln\Bigl(\frac{N}{2\eta}\Bigr)
\end{align*}
so that by \eqref{gammato18a} and \eqref{gammato18b}
\begin{equation*}
\sum_{n=[\frac{N}{\e}]+1}^\infty e^{S_n(\eta,\e)} \leq \frac{2\eta}{\e} e^{S_{[N/\e]}(\eta,\e)} \sum_{m=0}^\infty e^{-\ln(\frac{N}{2\eta})m}
\leq \frac{C\eta}{\e}e^{\frac{W(\eta)-\alpha}{\e}}e^{2N}
\end{equation*}
and also this term decays exponentially faster than the rest of the sum as $\e\to0$.

This shows that the asymptotics of \eqref{gammato17} is determined only by the terms of the series with $|n-n_*|\leq \sigma n_*$: for such values we can approximate $\bar{S}_n$ by its value at the maximum using \eqref{gammato18},
\begin{align*}
\Psi\Bigl(\frac{\eta}{\e}\Bigr)
& \sim \frac{1}{2\e^{\frac32}} \sum_{n=[n_*(1-\sigma)]}^{[n_*(1+\sigma)]} \frac{\sqrt{1-e^{-n\e}}}{\sqrt{n}(n+1)} e^{\bar{S}_n(\eta,\e)} e^{R_n(\e)} \\
& \sim \frac{\eta\sqrt{1-e^{-t_*(\eta)}}}{t_*(\eta)^{\frac32}} e^{R_{n_*}(\e)} e^{\frac{W(\eta)}{\e}} \sum_{n=[n_*(1-\sigma)]}^{[n_*(1+\sigma)]}  e^{-\e D(\eta)(n-n_*)^2}\,.
\end{align*}
We can now extend the region of summation to infinity without changing the asymptotics, and the resulting sum can be seen as a Riemann sum of a Gaussian integral:
\begin{align*}
\Psi\Bigl(\frac{\eta}{\e}\Bigr)
& \sim \frac{\eta\sqrt{1-e^{-t_*(\eta)}}}{t_*(\eta)^{\frac32}} e^{R_{n_*}(\e)} e^{\frac{W(\eta)}{\e}} \sum_{n=-\infty}^{\infty}  e^{-\e D(\eta)n^2} \\
& = \frac{\eta\sqrt{1-e^{-t_*(\eta)}}}{\e t_*(\eta)^{\frac32}} e^{R_{n_*}(\e)} e^{\frac{W(\eta)}{\e}} \sum_{n=-\infty}^{\infty}  e^{ -\frac{D(\eta)}{\e}(\e n)^2}\e \\
& \sim \frac{\eta\sqrt{1-e^{-t_*(\eta)}}}{\e t_*(\eta)^{\frac32}} e^{R_{n_*}(\e)} e^{\frac{W(\eta)}{\e}} \int_{-\infty}^{\infty} e^{ -\frac{D(\eta)}{\e}x^2}\de x \\
& = \frac{\eta\sqrt{\pi}\sqrt{1-e^{-t_*(\eta)}}}{\sqrt{\e} \sqrt{D(\eta)} t_*(\eta)^{\frac32}} e^{R_{n_*}(\e)} e^{\frac{W(\eta)}{\e}} \,.
\end{align*}
Summing up, we have obtained an asymptotics for the function $\Psi$ solving the linearized problem \eqref{gammato13} of the following form:
\begin{equation} \label{gammato19}
\Psi\Bigl(\frac{\eta}{\e}\Bigr) \sim \frac{U(\eta)}{\sqrt{\e}} e^{\frac{W(\eta)}{\e}} \qquad\text{as }\e\to0^+,
\end{equation}
where $\eta>\frac12$ is a parameter.

\bigskip

We now go back to the solution $H$ to the nonlinear problem \eqref{gammato11}. By \eqref{gammato19} we have
\begin{equation}\label{gammato110}
H(y) \sim 1 - \delta \Psi(y) \sim 1 - \frac{\delta}{\sqrt{\e}}U(\eta)e^{\frac{W(\eta)}{\e}}\,, \qquad y = \frac{\eta}{\e}\,.
\end{equation}
This asymptotic is valid until $\delta\Psi$ becomes of order 1. Suppose that the two parameters $\e$, $\delta$ appearing in \eqref{gammato11} satisfy a relation of the form
\begin{equation} \label{gammato111}
\delta = \frac{\sqrt{\e}}{U(\bar{\eta})} e^{-\frac{W(\bar{\eta})}{\e}}
\end{equation}
for some $\bar\eta>\frac12$ which now plays the role of a free parameter. Then for values of the variable $y$ close to the point $\bar{y}:=\frac{\bar{\eta}}{\e}$ we have by \eqref{gammato110}
\begin{align*}
H(y)
& \sim 1 - \frac{\delta}{\sqrt{\e}} U\bigl(\bar\eta + \e(y-\bar{y})\bigr) e^{\frac{W(\bar{\eta} + \e(y-\bar{y}))}{\e}} \\
& \sim 1 - \frac{\delta}{\sqrt{\e}} U(\bar{\eta}) e^{\frac{W(\bar{\eta})}{\e} + W'(\bar{\eta})(y-\bar{y})} \\
& = 1 - e^{W'(\bar{\eta})(y-\bar{y})}\,.
\end{align*}
Then in the transition region where $y\sim\bar{y}=\frac{\bar{\eta}}{\e}$ we can consider the translated function $h(z):=H(\bar{y}+z)$, whose behaviour is determined by the approximate equation
\begin{equation*}
h'(z) = - \bigl(h(z-\bar{\eta})\bigr)^2 + \bigl(h(z)\bigr)^2
\end{equation*}
(notice that we can neglect the term multiplied by the parameter $\delta$, which is exponentially small with respect to $\e$ by \eqref{gammato111}), with the matching condition
\begin{equation*}
h(z) \sim 1 - e^{W'(\bar{\eta})z} \qquad\text{as }z\to-\infty.
\end{equation*}
Notice that with the change of variables $\Phi(x)=h(\frac{\bar{\eta}\ln x}{\ln2})$ we obtain the equation
\begin{equation*}
\frac{\ln2}{\bar{\eta}} x \Phi'(x) = - \bigl(\Phi(x/2)\bigr)^2 + \bigl(\Phi(x)\bigr)^2\,,
\end{equation*}
which is nothing else but the equation for the self-similar profile in the limit case of homogeneity $\gamma=1$, see \eqref{gamma1a}.

\bigskip

We now want to study the transition from the region where $y\sim\bar{y}=\frac{\bar{\eta}}{\e}$ to the region $y\to\infty$. To this aim, we keep only the dominant terms in the equation, which are
\begin{equation} \label{gammato120}
H'(y) = - \bigl( H(y(1-\e)) \bigr)^2\,.
\end{equation}
We first consider the region where $1\ll y-\bar{y} \ll \frac{1}{\e}$; in this region we have $\e y \sim \e \bar{y}=\bar{\eta}$ and we approximate the equation by
\begin{equation} \label{gammato121}
H'(y) = - \bigl( H(y-\bar{\eta}) \bigr)^2\,.
\end{equation}
We can expect an asymptotics in the form
\begin{equation} \label{gammato122}
H(y) \sim K_0 e^{\sigma(y-\bar{y})} e^{-c_0e^{\sigma(y-\bar{y})}} \,.
\end{equation}
Indeed by inserting this ansatz in \eqref{gammato121} we obtain
\begin{equation*}
-K_0c_0\sigma e^{2\sigma(y-\bar{y})} e^{-c_0e^{\sigma(y-\bar{y})}} \sim - K_0^2 e^{-2\sigma\bar{\eta}} e^{2\sigma(y-\bar{y})} e^{-2c_0 e^{-\sigma\bar{\eta}}e^{\sigma(y-\bar{y})}}\,,
\end{equation*}
which yields the relations
\begin{equation*}
2e^{-\sigma\bar{\eta}} = 1, \qquad c_0\sigma = K_0 e^{-2\sigma\bar{\eta}}\,,
\end{equation*}
whence
\begin{equation} \label{gammato123}
\sigma = \frac{\ln2}{\bar{\eta}}, \qquad K_0 = \frac{4c_0\ln2}{\bar{\eta}}\,.
\end{equation}
The constant $c_0$ is the unique degree of freedom in the asymptotics \eqref{gammato122}.

We next compute the asymptotics for the full problem \eqref{gammato120} as $y\to\infty$: we expect an exponentially decaying solution and we hence make the ansatz
\begin{equation} \label{gammato124}
H(y) \sim K_1 y^\alpha e^{-c_1 y^{\beta}},
\end{equation}
which inserted in \eqref{gammato120} gives
\begin{equation*}
-K_1c_1\beta y^{\alpha+\beta-1}e^{-c_1y^\beta} \sim -K_1^2 (1-\e)^{2\alpha }y^{2\alpha} e^{-2c_1(1-\e)^\beta y^\beta},
\end{equation*}
from which we obtain the relations
\begin{equation*}
2(1-\e)^\beta = 1, \qquad \alpha+\beta-1=2\alpha, \qquad c_1\beta = K_1(1-\e)^{2\alpha},
\end{equation*}
whence
\begin{equation} \label{gammato125}
\beta = -\frac{\ln2}{\ln(1-\e)} \sim \frac{\ln2}{\e}, \qquad \alpha = \beta-1 \sim \frac{\ln2}{\e}, \qquad K_1 = 4c_1\beta(1-\e)^2 \sim\frac{4c_1\ln2}{\e}\,.
\end{equation}
Again the only degree of freedom of the solution is the constant $c_1$.

We finally examine the transition from the first asymptotics \eqref{gammato122} to the second \eqref{gammato124}, and we show that the two asymptotics can be connected: indeed, in the region where the approximation \eqref{gammato121} is valid, that is for $1\ll y-\bar{y} \ll \frac{1}{\e}$, we can express \eqref{gammato124} in the following form, by writing $y=\bar{y} + (y-\bar{y})$:
\begin{align*}
K_1 y^\alpha e^{-c_1 y^{\beta}}
& = K_1 \bigl(\bar{y}+(y-\bar{y})\bigr)^\alpha e^{-c_1 (\bar{y}+(y-\bar{y}))^{\beta}} \\
& = K_1 \bar{y}^\alpha \Bigl( 1 + \frac{y-\bar{y}}{\bar{y}}\Bigr)^\alpha \exp\Bigl(-c_1\bar{y}^\beta\Bigl(1+\frac{y-\bar{y}}{\bar{y}}\Bigr)^\beta\Bigr) \\
& \approx K_1\Bigl(\frac{\bar{\eta}}{\e}\Bigr)^\alpha e^{\frac{\alpha}{\bar{y}}(y-\bar{y})} \exp\Bigl( -c_1\Bigl(\frac{\bar{\eta}}{\e}\Bigr)^\beta e^{\frac{\beta}{\bar{y}}(y-\bar{y})} \Bigr) \,.
\end{align*}
Here we used the approximations $(1+\frac{y-\bar{y}}{\bar{y}})^\alpha\approx e^{\frac{\alpha}{\bar{y}}(y-\bar{y})}$, $(1+\frac{y-\bar{y}}{\bar{y}})^\beta\approx e^{\frac{\beta}{\bar{y}}(y-\bar{y})}$, which follow from the fact that $\bar{y}=\frac{\bar{\eta}}{\e}$ and that both $\alpha$ and $\beta$ are of order $\frac{1}{\e}$. Moreover, as $\alpha=\beta-1=\frac{\ln2}{\e}(1+O(\e))$, in the region where $\e(y-\bar{y})$ is small we also have
$$
e^{\frac{\alpha}{\bar{y}}(y-\bar{y})}=e^{\frac{\beta}{\bar{y}}(y-\bar{y})}(1+O(\e)) = e^{\frac{\ln2}{\bar{\eta}}(y-\bar{y})}(1+O(\e)).
$$
Therefore
\begin{align} \label{gammato126}
K_1 y^\alpha e^{-c_1 y^{\beta}}
& \approx K_1\Bigl(\frac{\bar{\eta}}{\e}\Bigr)^\alpha e^{\frac{\ln2}{\bar{\eta}}(y-\bar{y})} \exp\Bigl( -c_1\Bigl(\frac{\bar{\eta}}{\e}\Bigr)^\beta e^{\frac{\ln2}{\bar{\eta}}(y-\bar{y})} \Bigr) \,.
\end{align}
By comparing the right-hand side of \eqref{gammato126} with the first asymptotics \eqref{gammato122}, and recalling that $\sigma=\frac{\ln2}{\bar{\eta}}$, we see that the two expressions agree, up to identifying 
\begin{equation} \label{gammato127}
K_0 = K_1 \Bigl(\frac{\bar{\eta}}{\e}\Bigr)^\alpha\,, \qquad c_0 = c_1 \Bigl(\frac{\bar{\eta}}{\e}\Bigr)^\beta\,.
\end{equation}
We only need to check that the two conditions $K_0=\frac{4c_0\ln2}{\bar{\eta}}$ and $K_1\sim\frac{4c_1\ln2}{\e}$, obtained in \eqref{gammato123} and \eqref{gammato125} respectively, are valid with the positions \eqref{gammato127}: indeed by using the first two conditions in \eqref{gammato125} we obtain
\begin{align*}
K_1 = K_0\Bigl(\frac{\e}{\bar{\eta}}\Bigr)^\alpha = \frac{4c_0\ln2}{\bar{\eta}}\Bigl(\frac{\e}{\bar{\eta}}\Bigr)^\alpha = \frac{4c_1\ln2}{\bar{\eta}}\Bigl(\frac{\e}{\bar{\eta}}\Bigr)^{\alpha-\beta}
= \frac{4c_1\ln2}{\e}\,,
\end{align*}
and the two asymptotics are therefore equivalent, with the value of $c_1$ obtained from the value of $c_0$ by means of $c_1=c_0(\e/\bar{\eta})^\beta$.

\bigskip

In conclusion, for small $\e$ there is a critical value of the parameter $\delta$, which is exponentially small in $\e$ and is precisely given by the relation \eqref{gammato111}, such that the asymptotics of the solution $H$ to \eqref{gammato11} for large values $y\to\infty$ is given by \eqref{gammato124}.


\section{Rigorous proof in the case \texorpdfstring{\boldmath$\gamma\to\infty$}{of large homogeneity}} \label{sect:gammatoinf}

The existence of a critical value of the parameter $b$ for which there is a positive solution to \eqref{equation} with exponential decay as $x\to\infty$ can be proved rigorously for large values of the homogeneity $\gamma$. The main result of this section is the following.

\begin{theorem}[The case $\gamma\to\infty$] \label{thm:gammatoinf}
	There exists $\bar{\gamma}>1$ with the following property. For every $\gamma>\bar{\gamma}$ there exist $\bar{b}(\gamma)>0$ and a positive solution $\Phi$ to \eqref{equation} with exponential decay.
\end{theorem}

In order to prove the theorem, we again consider the formulation \eqref{equation2} of the equation and we rescale the solution by the parameter $\sigma$, as was done in Section~\ref{subsect:oscill}: by setting $h(x):= H(\frac{x}{\sigma})$, we have that $h$ solves
\begin{equation} \label{gammatoinf0}
\begin{cases}
h'(x) = -\bigl(h\bigl(\frac{x}{2}(1+\e)\bigr)\bigr)^2 + \eta\bigl(h(x)\bigr)^2\,,\\
h(0)=1,
\end{cases}
\end{equation}
where we also introduced two small parameters $\e$, $\eta$ given by
\begin{equation} \label{parameters2}
2^\frac1b=\frac{2}{1+\e}, \qquad \eta = \frac{1}{\sigma} = 2^{\frac{2}{b}+1-\gamma}.
\end{equation}
We will prove Theorem~\ref{thm:gammatoinf} by showing that for every $\eta>0$ sufficiently small, there is a unique value of the parameter $\e=\e(\eta)>0$ such that the corresponding solution $h(x;\e(\eta),\eta)$ to \eqref{gammatoinf0} satisfies
\begin{equation*}
h(x;\e(\eta),\eta)>0, \qquad \liminf_{x\to\infty} \Bigl(x\,h(x;\e(\eta),\eta)\Bigr)=0
\end{equation*}
(and actually decays exponentially fast).

\begin{remark}
	One can get an insight into the behaviour of the solution to \eqref{gammatoinf0} by looking at the limit problem for $\eta=0$, which has been discussed in detail in Section~\ref{subsect:oscill}. For $\e=0$ the problem has the explicit solution $\bar{h}(x)=e^{-x}$. Moreover, the solutions corresponding to $\e<0$ cross to negative values, while for $\e>0$ one has $\liminf_{x\to\infty} \bigl(x\,h(x;\e(\eta),\eta)\bigr)>0$.
\end{remark}

The same picture outlined in the previous remark should remain valid also for small values of $\eta$. The strategy of the proof is the following: we look for a solution for two given values $\eta$, $\e$ of the parameters in the form
\begin{equation} \label{gammatoinf0b}
h(x;\e,\eta)=\bar{h}(x)+W(x;\e,\eta)\,,
\end{equation}
where $\bar{h}(x)=e^{-x}$ is the explicit solution corresponding to $\eta=\e=0$. By substituting the ansatz \eqref{gammatoinf0b} in \eqref{gammatoinf0} we find that $W$ satisfies
\begin{align*}
W_x(x;\e,\eta) = -2e^{-\frac{x}{2}}W(\textstyle\frac{x}{2};\e,\eta) + R[W](x;\e,\eta), \qquad W(0;\e,\eta)=0,
\end{align*}
where
\begin{equation} \label{gammatoinf3bis}
\begin{split}
R[W](x;\e,\eta) &:= e^{-x}-e^{-x(1+\e)} + 2e^{-\frac{x}{2}}W(\textstyle\frac{x}{2};\e,\eta) - 2e^{-\frac{x}{2}(1+\e)}W(\textstyle\frac{x}{2}(1+\e);\e,\eta) \\
& \qquad - \bigl(W(\textstyle\frac{x}{2}(1+\e);\e,\eta)\bigr)^2 + \eta\bigl(e^{-x}+W(x;\e,\eta)\bigr)^2\,.
\end{split}
\end{equation}
The function $\vphi(x;\e,\eta)=e^xW(x;\e,\eta)$ therefore solves the linear problem
\begin{equation*}
\vphi_x(x;\e,\eta) = \vphi(x;\e,\eta) - 2\vphi(\textstyle\frac{x}{2};\e,\eta) + e^{x}R[W](x;\e,\eta)\,, \qquad x>0,
\end{equation*}
with $\vphi(0;\e,\eta)=0$. This delay equation, where the nonlinearity $R[W]$ is treated as a source term, is studied in details in Section~\ref{sect:linear}; we can write a representation formula for $\vphi$ by using the fundamental solution obtained in Lemma~\ref{lem:linearproblem}:
\begin{equation*}
\begin{split}
W(x;\e,\eta) &= e^{-x}\vphi(x;\e,\eta) = e^{-x}\int_0^x G(x,\xi)e^\xi R[W](\xi;\e,\eta)\de\xi \\
& = \int_0^x e^\xi Q(\xi) R[W](\xi;\e,\eta)\de\xi + \int_0^x e^{\xi-x} \widetilde{G}(x,\xi) R[W](\xi;\e,\eta)\de \xi\,.
\end{split}
\end{equation*}
If we introduce the quantity
\begin{equation} \label{gammatoinf2}
F(W,\e,\eta) := \int_0^\infty e^\xi Q(\xi) R[W](\xi;\e,\eta)\de\xi
\end{equation}
we can rewrite the equation for $W$ as follows:
\begin{equation} \label{gammatoinf1}
W(x;\e,\eta) = F(W,\e,\eta) -\int_x^\infty e^{\xi} Q(\xi) R[W](\xi;\e,\eta)\de\xi  + \int_0^x e^{\xi-x} \widetilde{G}(x,\xi) R[W](\xi;\e,\eta)\de \xi\,.
\end{equation}
The value $F(W,\e,\eta)$ represents the constant value of $W$ at infinity. The goal is now to show that for every sufficiently small value of the parameters $\e$ and $\eta$ there is a unique solution to the equation
\begin{equation} \label{gammatoinf3}
W(x;\e,\eta) = - \int_x^\infty e^{\xi} Q(\xi) R[W](\xi;\e,\eta)\de\xi  + \int_0^x e^{\xi-x} \widetilde{G}(x,\xi) R[W](\xi;\e,\eta)\de \xi\,,
\end{equation}
which will be obtained by means of a fixed point argument (Lemma~\ref{lem:gammatoinfFP}), and furthermore to prove that for every $\eta$ small enough we can find a unique $\e=\e(\eta)$ such that the corresponding solution satisfies $F(W(\cdot;\e(\eta),\eta),\e(\eta),\eta)=0$ and is positive (Lemma~\ref{lem:gammatoinfIFT}).

\begin{lemma}[Existence by fixed point] \label{lem:gammatoinfFP}
	There exist $\e_0>0$ and $\eta_0>0$ such that for every $\e\in[0,\e_0]$ and $\eta\in[0,\eta_0]$ the equation \eqref{gammatoinf3} has a unique solution $W(\cdot;\e,\eta)\in C^0([0,\infty))$ with
	\begin{equation} \label{gammatoinf4bis}
	|W(x;\e,\eta)| \leq (\e+\eta) M e^{-\delta x}
	\end{equation}
	for constants $\delta>0$ and $M>0$ independent of $\e$ and $\eta$.
\end{lemma}

\begin{proof}
	Let $\delta\in(0,\beta)$, where $\beta\in(0,\frac12)$ is as in Lemma~\ref{lem:linearproblem}, and let $M>1$ be a fixed constant, to be chosen later independently of $\e$ and $\eta$ (possibly depending on $\delta$). We can assume without loss of generality $(\e+\eta)M^2\leq1$. We introduce the space
	\begin{equation*}
	X_{\e,\eta} := \Bigl\{ W\in C^0([0,\infty)) \;:\; |W(x)| \leq (\e+\eta) M e^{-\delta x} \Bigr\}
	\end{equation*}
	with norm
	\begin{equation*}
	\|W\| := \sup_{x\geq0} \frac{|W(x)|}{e^{-\delta x}}\,.
	\end{equation*}
	We define an operator $T_{\e,\eta}$ acting on $W\in X_{\e,\eta}$ by setting
	\begin{align} \label{gammatoinf4}
	T_{\e,\eta}[W](x) := -\int_x^\infty e^\xi Q(\xi)R_{\e,\eta}[W](\xi)\de \xi + \int_0^x e^{\xi-x} \widetilde{G}(x,\xi) R_{\e,\eta}[W](\xi)\de \xi\,,
	\end{align}
	where
	\begin{equation} \label{gammatoinf5}
	\begin{split}
	R_{\e,\eta}[W](x) &:= e^{-x}-e^{-x(1+\e)} + 2e^{-\frac{x}{2}}W(\textstyle\frac{x}{2}) - 2e^{-\frac{x}{2}(1+\e)}W(\textstyle\frac{x}{2}(1+\e)) \\
	& \qquad - \bigl(W(\textstyle\frac{x}{2}(1+\e))\bigr)^2 + \eta\bigl(e^{-x}+W(x)\bigr)^2\,.
	\end{split}
	\end{equation}
	We claim that for $\e$ and $\eta$ sufficiently small $T_{\e,\eta}$ is a contraction on $X_{\e,\eta}$, so that Banach's Fixed Point Theorem yields the existence of a unique $W\in X_{\e,\eta}$ satisfying $T_{\e,\eta}[W]=W$, which is equivalent to \eqref{gammatoinf3}.
	
	We first show that $T_{\e,\eta}$ maps $X_{\e,\eta}$ into itself. It is convenient to split the term $R_{\e,\eta}$ into two parts, which will be treated separately:
	\begin{align*}
	R_{\e,\eta}^1[W](x) &:= e^{-x}-e^{-x(1+\e)} - \bigl(W(\textstyle\frac{x}{2}(1+\e))\bigr)^2 + \eta\bigl(e^{-x}+W(x)\bigr)^2\,, \\
	R_{\e,\eta}^2[W](x) &:= 2e^{-\frac{x}{2}}W(\textstyle\frac{x}{2}) - 2e^{-\frac{x}{2}(1+\e)}W(\textstyle\frac{x}{2}(1+\e))\,.
	\end{align*}
	Given any $W\in X_{\e,\eta}$, we estimate the quantity $R_{\e,\eta}^1[W]$ by means of the bound in the definition of the space $X_{\e,\eta}$:
	\begin{align} \label{gammatoinf5b}
	|R_{\e,\eta}^1[W](x)|
	& \leq \e xe^{-x} + |W(\textstyle\frac{x}{2}(1+\e))|^2 + 2\eta e^{-2x} + 2\eta |W(x)|^2 \nonumber\\
	& \leq \e x e^{-x} + (\e+\eta)^2M^2 e^{-\delta x(1+\e)} + 2\eta e^{-2x} + 2\eta(\e+\eta)^2M^2e^{-2\delta x} \nonumber\\
	& \leq C_1 (\e+\eta) e^{- \delta x}
	\end{align}
	for a numerical constant $C_1>0$, independent of $\e$, $\eta$, $\delta$, and $M$ (here we also used the assumption $(\e+\eta)M^2\leq1$).
	By plugging this estimate into the definition \eqref{gammatoinf4} of $T_{\e,\eta}$ and using the bounds \eqref{linearbounds1}--\eqref{linearbounds2} we get
	\begin{align}  \label{gammatoinf6}
	|T_{\e,\eta}^1[W](x)|
	& = \bigg| - \int_x^\infty e^\xi Q(\xi)R_{\e,\eta}^1[W](\xi)\de \xi + \int_0^x e^{\xi-x} \widetilde{G}(x,\xi) R_{\e,\eta}^1[W](\xi)\de \xi \bigg| \nonumber \\
	& \leq (\e+\eta)C_0 C_1 \biggl( \int_x^\infty e^{- \delta \xi}\de\xi + \int_0^x e^{\beta(\xi-x)} e^{- \delta\xi}\de\xi \biggr) \nonumber \\
	& \leq (\e+\eta)C_0 C_1 \Bigl( \frac{1}{\delta}+\frac{1}{\beta-\delta} \Bigr) e^{-\delta x}\,.
	\end{align}
	We now prove a similar bound on the part of the operator involving $R_{\e,\eta}^2[W]$. By a change of variable we have
	\begin{align} \label{gammatoinf7}
	\bigg| \int_x^\infty e^\xi Q(\xi)R^2_{\e,\eta}[W](\xi)\de \xi \bigg|
	& \leq 2\int_x^\infty \Big| e^\xi Q(\xi) - \textstyle\frac{1}{1+\e}e^{\frac{\xi}{1+\e}} Q(\frac{\xi}{1+\e}) \Big| e^{-\frac{\xi}{2}}|W(\frac{\xi}{2})|\de\xi \nonumber\\
	& \qquad + \frac{2}{1+\e}\int_{x}^{x(1+\e)} e^{\frac{\xi}{1+\e}} \textstyle |Q(\frac{\xi}{1+\e})| e^{-\frac{\xi}{2}}|W(\frac{\xi}{2})|\de\xi \nonumber\\
	& \leq 2(\e+\eta)M C_0 \biggl( 2\e \int_x^\infty e^{-\frac{\xi}{2}}e^{-\frac12\delta \xi}\de\xi + \int_{x}^{x(1+\e)} e^{-\frac{\xi}{2}}e^{-\frac12\delta \xi}\de\xi \biggr) \nonumber \\
	& \leq \e (\e+\eta)M C_0 C_1 \bigl(1+x) e^{-(\frac12+\frac{\delta}{2})x}
	\leq \e C_0 C_1 e^{-\delta x}
	\end{align}
	where the second inequality follows by using \eqref{linearbounds1} and the bound on $W$ given by the definition of the space $X_{\e,\eta}$, and $C_1$ is a numerical constant as before (possibly different). Similarly, using \eqref{linearbounds2},
	\begin{align} \label{gammatoinf8}
	\bigg| \int_0^x e^{\xi-x}\widetilde{G}(x,\xi) R^2_{\e,\eta}[W](\xi)\de \xi \bigg|
	& \leq 2e^{-x}\int_0^x \Big| e^\xi \widetilde{G}(x,\xi) - \textstyle\frac{1}{1+\e}e^{\frac{\xi}{1+\e}} \widetilde{G}(x,\frac{\xi}{1+\e}) \Big| e^{-\frac{\xi}{2}}|W(\frac{\xi}{2})|\de\xi \nonumber\\
	& \qquad + \frac{2e^{-x}}{1+\e}\int_{x}^{x(1+\e)} e^{\frac{\xi}{1+\e}} \textstyle |\widetilde{G}(x,\frac{\xi}{1+\e})| e^{-\frac{\xi}{2}}|W(\frac{\xi}{2})|\de\xi \nonumber\\
	& \leq 2(\e+\eta)M C_0 \biggl( \e \int_0^x (1+2\xi)e^{\beta(\xi-x)}e^{-\frac{\xi}{2}}e^{-\frac{\delta \xi}{2}}\de\xi \nonumber\\
	& \qquad \qquad + \int_{x}^{x(1+\e)} e^{\beta(\xi-x)}e^{-\frac{\xi}{2}}e^{-\frac{\delta \xi}{2}}\de\xi \biggr) \nonumber \\
	& \leq 2 C_0 \biggl( 10\e\int_0^x e^{\beta(\xi-x)}e^{-\delta \xi}\de\xi + \e x e^{\e\beta x}e^{-\frac{x}{2}}e^{-\frac{\delta x}{2}} \biggr) \nonumber \\
	& \leq 2\e C_0 \Bigl(\frac{10}{\beta-\delta} + xe^{(\e\beta+\frac{\delta}{2}-\frac12)x} \Bigr)e^{-\delta x} \leq \e C_0C_{\delta,\beta}e^{-\delta x}\,.
	\end{align}
	Therefore, combining \eqref{gammatoinf6}, \eqref{gammatoinf7}, and \eqref{gammatoinf8}, and by choosing $M>C_0C_1\bigl( 1+ \frac{1}{\delta}+\frac{1}{\beta-\delta} + C_{\delta,\beta} \bigr)$, we obtain
	\begin{equation} \label{gammatoinf9}
	|T_{\e,\eta}[W](x)| \leq (\e+\eta)M e^{-\delta x}
	\end{equation}
	for every $W\in X_{\e,\eta}$.
	
	\medskip
	We now show the continuity of $T_{\e,\eta}[W]$. Notice that, by using \eqref{gammatoinf5b} and the bound on $W$ in the definition of the space $X_{\e,\eta}$, it is straightforward to obtain the simple bound
	\begin{equation} \label{gammatoinf9b}
	|R_{\e,\eta}[W](x)| \leq C_1(\e+\eta)e^{-\delta x} + 4(\e+\eta)M e^{-\frac{x}{2}}e^{-\frac{\delta x}{2}} \leq C_2 e^{-\delta x}
	\end{equation}
	for a numerical constant $C_2>0$.
	Therefore for every pair of points $x_1<x_2$ we have using this estimate, \eqref{linearbounds1}, \eqref{linearbounds2}, and \eqref{linearbounds3},
	\begin{align*}
	|T_{\e,\eta}[W](x_1)-T_{\e,\eta}[W](x_2)|
	& \leq \int_{x_1}^{x_2} \Bigl( e^\xi |Q(\xi)| + e^{\xi-x_2}|\widetilde{G}(x_2,\xi)| \Bigr) |R_{\e,\eta}[W](\xi)| \de \xi \nonumber \\
	& \qquad + \int_0^{x_1} \big| e^{\xi-x_1} \widetilde{G}(x_1,\xi) - e^{\xi-x_2}\widetilde{G}(x_2,\xi) \big| |R_{\e,\eta}[W](\xi)| \de \xi \nonumber \\
	& \leq C_0C_2\int_{x_1}^{x_2} \Bigl( 1 + e^{\beta(\xi-x_2)} \Bigr) e^{-\delta\xi} \de\xi \nonumber\\
	& \qquad + C_0C_2 |x_1-x_2| \int_0^{x_1} \Bigl( e^{\beta(\xi-x_1)}+e^{\beta(\xi-x_2)} \Bigr) e^{-\delta\xi}\de\xi \nonumber\\
	& \leq 4C_0C_2 |x_1-x_2| \,.
	\end{align*}
	It follows that $T_{\e,\eta}[W]\in C^0([0,\infty))$, and this information combined with the bound \eqref{gammatoinf9} proves that $T_{\e,\eta}[W]\in X_{\e,\eta}$ for every $W\in X_{\e,\eta}$.
	
	\medskip
	We are left with the proof that $T_{\e,\eta}$ is a contraction in $X_{\e,\eta}$, for $\e$ and $\eta$ small enough. As before it is convenient to separate the contributions of the two parts $R_{\e,\eta}^1[W]$ and $R_{\e,\eta}^2[W]$ of $R_{\e,\eta}[W]$. For the first, it is straightforward to obtain the bound
	\begin{align*}
	|R_{\e,\eta}^1[W_1](x)- R_{\e,\eta}^1[W_2](x)|
	& \leq \big| \bigl(W_1(\textstyle\frac{x}{2}(1+\e))\bigr)^2 - \bigl(W_2(\textstyle\frac{x}{2}(1+\e))\bigr)^2 \big| \nonumber \\
	& \qquad + \eta \Bigl( (e^{-x}+W_1(x))^2 - (e^{-x}+W_2(x))^2 \Bigr) \nonumber \\
	& \leq C_3(\e+\eta)\|W_1-W_2\|e^{-\delta x}
	\end{align*}
	for every $W_1,W_2\in X_{\e,\eta}$, for some constant $C_3>0$ depending possibly only on $M$. By inserting this estimate in the definition \eqref{gammatoinf4} of $T_{\e,\eta}[W]$ and arguing as in \eqref{gammatoinf6}, we find
	\begin{equation} \label{gammatoinf10}
	|T_{\e,\eta}^1[W_1](x)-T_{\e,\eta}^1[W_2](x)| \leq (\e+\eta)C_0C_3\Bigl(\frac{1}{\delta}+\frac{1}{\beta-\delta}\Bigr)\|W_1-W_2\|e^{-\delta x}\,.
	\end{equation}
	We now consider the part of the operator involving $R_{\e,\eta}^2[W]$. By a change of variables and arguing as in \eqref{gammatoinf7} we have
	\begin{align}\label{gammatoinf11}
	\bigg| \int_x^\infty e^\xi Q(\xi) & \Bigl( R^2_{\e,\eta}[W_1](\xi) - R^2_{\e,\eta}[W_2](\xi) \Bigr) \de \xi \bigg| \nonumber \\
	& \leq 2\int_x^\infty \Big| e^\xi Q(\xi) - \textstyle\frac{1}{1+\e}e^{\frac{\xi}{1+\e}} Q(\frac{\xi}{1+\e}) \Big| e^{-\frac{\xi}{2}} \big| W_1(\frac{\xi}{2})-W_2(\frac{\xi}{2}) \big|\de\xi \nonumber\\
	& \qquad + \frac{2}{1+\e}\int_{x}^{x(1+\e)} e^{\frac{\xi}{1+\e}} \textstyle |Q(\frac{\xi}{1+\e})| e^{-\frac{\xi}{2}} \big|W_1(\frac{\xi}{2})-W_2(\frac{\xi}{2}) \big| \de\xi \nonumber\\
	& \leq \e C_0 C_1 \|W_1-W_2\| e^{-\delta x}\,.
	\end{align}
	Similarly by the same estimates as in \eqref{gammatoinf8}
	\begin{align} \label{gammatoinf12}
	\bigg| \int_0^x e^{\xi-x}\widetilde{G}(x,\xi) & \Bigl( R^2_{\e,\eta}[W_1](\xi) - R^2_{\e,\eta}[W_2](\xi) \Bigr)\de \xi \bigg| \nonumber\\
	& \leq 2e^{-x}\int_0^x \Big| e^\xi \widetilde{G}(x,\xi) - \textstyle\frac{1}{1+\e}e^{\frac{\xi}{1+\e}} \widetilde{G}(x,\frac{\xi}{1+\e}) \Big| e^{-\frac{\xi}{2}} \big| W_1(\frac{\xi}{2}) - W_2(\frac{\xi}{2}) \big| \de\xi \nonumber\\
	& \qquad + \frac{2e^{-x}}{1+\e}\int_{x}^{x(1+\e)} e^{\frac{\xi}{1+\e}} \textstyle |\widetilde{G}(x,\frac{\xi}{1+\e})| e^{-\frac{\xi}{2}} \big| W_1(\frac{\xi}{2}) - W_2(\frac{\xi}{2}) \big| \de\xi \nonumber\\
	& \leq \e C_0 C_{\delta,\beta} \|W_1-W_2\| e^{-\delta x}\,.
	\end{align}
	Combining \eqref{gammatoinf10}, \eqref{gammatoinf11} and \eqref{gammatoinf12} we finally obtain
	\begin{equation*}
	\|T_{\e,\eta}[W_1]-T_{\e,\eta}[W_2]\| \leq (\e+\eta) C \|W_1-W_2\|\,,
	\end{equation*}
	for every $W_1,W_2\in X_{\e,\eta}$ and for a constant $C>0$ independent of $\e$ and $\eta$. Therefore the map $T_{\e,\eta}$ is a contraction for $\e$ and $\eta$ sufficiently small.
\end{proof}

\begin{remark}[Differentiability of $W$]\label{rm:gammatoinfC1}
	We can now check that the solution $W(\cdot;\e,\eta)$ constructed in Lemma~\ref{lem:gammatoinfFP} is continuously differentiable. Indeed, we can rewrite \eqref{gammatoinf3} in the following form:
	\begin{align*}
	W(x;\e,\eta) = - \int_0^\infty e^{\xi} Q(\xi) R[W](\xi;\e,\eta)\de\xi  + \int_0^x e^{\xi-x} G(x,\xi) R[W](\xi;\e,\eta)\de \xi
	\end{align*}
	(notice that, in view of the bound \eqref{gammatoinf4bis} on $W$, the function $R[W]$ decays exponentially and the previous integrals are therefore well-defined quantities). By taking the derivative of the previous expression with respect to the variable $x$ we find, using that $G$ is the fundamental solution to \eqref{linearproblem},
	\begin{align} \label{gammatoinfC1}
	W_x(x;\e,\eta)
	& = R[W](x;\e,\eta) -2 \int_0^x e^{\xi-x} G({\textstyle\frac{x}{2}},\xi) R[W](\xi;\e,\eta) \de\xi \,.
	\end{align}
	From the definition \eqref{gammatoinf3bis} of $R[W]$ and the continuity of $W$ and $G$ it follows that the right-hand side in \eqref{gammatoinfC1} is continuous in $x$. Therefore $W(\cdot;\e,\eta)\in C^1([0,\infty))$.
	
	We can in addition obtain a bound on $W_x$ as follows: it was shown along the proof of Lemma~\ref{lem:gammatoinfFP}, using the estimate \eqref{gammatoinf4bis} on $W$, that
	\begin{equation*}
	|R[W](x;\e,\eta)| \leq (\e+\eta)Ce^{-\delta x}
	\end{equation*}
	for some constant $C$ depending only on $M$ (see \eqref{gammatoinf9b}); in turn, inserting this estimate into \eqref{gammatoinfC1} one finds, recalling the bounds on the function $G$ proved in Lemma~\ref{lem:linearproblem},
	\begin{equation*}
	|W_x(x;\e,\eta)| \leq (\e+\eta)M_1e^{-\delta x}\,,
	\end{equation*}
	for a uniform constant $M_1$ independent of $\e$ and $\eta$ ($M_1$ depends ultimately on $\delta$, $M$, and on the constants $\beta$ and $C_0$ appearing in Lemma~\ref{lem:linearproblem}).
\end{remark}

\begin{lemma}[Implicit Function Theorem] \label{lem:gammatoinfIFT}
	There are $\eta_1\in(0,\eta_0)$ and $\e_1\in(0,\e_0)$ with the following property: for every $\eta\in(0,\eta_1)$ there exists a unique $\e(\eta)\in(0,\e_1)$ such that the solution $W(\cdot;\e(\eta),\eta)$ to \eqref{gammatoinf3} constructed in Lemma~\ref{lem:gammatoinfFP} satisfies
	\begin{equation} \label{gammatoinf3b}
	F(W(\cdot;\e(\eta),\eta),\e(\eta),\eta)=0\,,
	\end{equation}
	where $F$ is defined in \eqref{gammatoinf2}.
\end{lemma}

\begin{proof}
Define the function $f(\e,\eta):= F(W(\cdot;\e,\eta),\e,\eta)$. The proof will follow by applying a one-side version of the Implicit Function Theorem (see Lemma~\ref{lem:one-sideIFT} below) to the function $f$ in a neighborhood of the point $(\e,\eta)=(0,0)$. Notice that $f(0,0)=0$, as $W(\cdot;0,0)\equiv0$. The main technical part in the proof consists therefore in proving that $f$ is continuously differentiable in the two parameters.

Recall that by the construction in Lemma~\ref{lem:gammatoinfFP} and the subsequent Remark~\ref{rm:gammatoinfC1} the function $W(\cdot;\e,\eta)\in C^1([0,\infty))$ satisfies the equation \eqref{gammatoinf3}, with $R[W]$ defined in \eqref{gammatoinf3bis}, together with the estimates
\begin{equation} \label{gammatoinf21}
|W(x;\e,\eta)| \leq (\e+\eta)M e^{-\delta x}\,, \qquad |W_x(x;\e,\eta)| \leq (\e+\eta)M_1 e^{-\delta x}\,,
\end{equation}
where $M,M_1>1$ and $\delta\in(0,\beta)$ are fixed constants and $\beta$ is as in Lemma~\ref{lem:linearproblem}.
Along the proof we will denote by $C$ a uniform, positive constant depending possibly only on the fixed parameters $M$, $M_1$, $\delta$, $\beta$, but not on $\e$ or $\eta$, which might change from line to line.

\medskip
\noindent\textit{Step 1: Lipschitz continuity of $W$ in $\e$ and $\eta$.}
We first prove a uniform bound on the difference quotients of $W$ in the variable $\e$. To this aim, we fix $(\e,\eta)\in[0,\frac{\e_1}{2})\times[0,\frac{\eta_1}{2})$, where $\e_1<\e_0$ and $\eta_1<\eta_0$ are to be chosen later. In this argument the variable $\eta$ will always take a fixed value and therefore we will not indicate the dependence on $\eta$ to lighten the notation. Fix also $h_0>0$, and define the quantity
\begin{equation} \label{gammatoinf22}
S_{h_0} := \sup_{h\in(h_0,\frac{\e_1}{2})} \, \sup_{x>0} \, \bigg| \frac{W(x;\e+h)-W(x;\e)}{h} \bigg| + \bigg| \frac{W_x(x;\e+h)-W_x(x;\e)}{h} \bigg|
\end{equation}
(which is finite in view of the bound \eqref{gammatoinf21}).

We introduce to simplify the notation the function
\begin{align} \label{gammatoinf22b}
\rho(x,\e) &:= R[W(\cdot;\e)](x;\e) = e^{-x}-e^{-x(1+\e)} + 2e^{-\frac{x}{2}}W(\textstyle\frac{x}{2};\e) - 2e^{-\frac{x}{2}(1+\e)}W(\textstyle\frac{x}{2}(1+\e);\e) \nonumber \\
& \qquad - \bigl(W(\textstyle\frac{x}{2}(1+\e);\e)\bigr)^2 + \eta\bigl(e^{-x}+W(x;\e)\bigr)^2\,,
\end{align}
so that by \eqref{gammatoinf3}
\begin{equation} \label{gammatoinf23}
\begin{split}
\frac{|W(x;\e+h)-W(x;\e)|}{h}
& \leq \int_x^\infty e^{\xi}|Q(\xi)| \frac{|\rho(\xi,\e+h)-\rho(\xi,\e)|}{h} \de\xi \\
& \qquad + \int_0^x e^{\xi-x}|\widetilde{G}(x,\xi)| \frac{|\rho(\xi,\e+h)-\rho(\xi,\e)|}{h} \de\xi\,.
\end{split}
\end{equation}

The first step is to obtain a uniform bound for $h_0<h<\frac{\e_1}{2}$ on the difference quotients of the function $\rho$ with respect to the variable $\e$: we can write
\begin{align*}
\bigg| \frac{\rho(x,\e+h)-\rho(x,\e)}{h} \bigg|
& \leq \frac{|e^{-x(1+\e+h)}-e^{-x(1+\e)}|}{h} \\
& \qquad + \frac{2e^{-\frac{x}{2}}}{h}  \Big| W(\textstyle\frac{x}{2};\e+h) - e^{-\frac{x}{2}(\e+h)}W(\frac{x}{2}(1+\e+h);\e+h) \\
& \qquad\qquad -W(\textstyle\frac{x}{2};\e) +e^{-\frac{x}{2}\e} W(\frac{x}{2}(1+\e);\e) \Big|\\
& \qquad + \frac{1}{h} \Big| \bigl(W(\textstyle\frac{x}{2}(1+\e+h);\e+h)\bigr)^2 - \bigl(W(\frac{x}{2}(1+\e);\e)\bigr)^2  \Big| \\
& \qquad +\frac{\eta}{h} \Big| \bigl( e^{-x} + W(x;\e+h) \bigr)^2 - \bigl( e^{-x} + W(x;\e) \bigr)^2  \Big| \\
& =: A_1 + A_2 + A_3 + A_4\,.
\end{align*}
We now estimate separately each term $A_i$. For the first one we have (recall that $\delta<\frac12$)
\begin{align*}
A_1 \leq x e^{-(1+\e)x} \leq e^{- \delta x}\,.
\end{align*}
The term $A_2$ can be written in the following form:
\begin{align*}
A_2 &\leq \frac{2e^{-\frac{x}{2}}}{h} \bigg|
\textstyle e^{-\frac{x}{2}(\e+h)} \Bigl( W(\frac{x}{2}(1+\e+h);\e+h) - W(\frac{x}{2}(1+\e);\e+h) \Bigr) \\
& \textstyle\qquad\qquad + \bigl( e^{-\frac{x}{2}(\e+h)}-e^{-\frac{x}{2}\e} \bigr) W(\frac{x}{2}(1+\e);\e+h) \\
& \textstyle\qquad\qquad + \bigl( e^{-\frac{x}{2}\e}-1\bigr) \Bigl( W(\frac{x}{2}(1+\e);\e+h) - W(\frac{x}{2}(1+\e);\e) \Bigr) \\
& \textstyle\qquad\qquad + \Bigl( W(\frac{x}{2}(1+\e);\e+h) - W(\frac{x}{2};\e+h) - W(\frac{x}{2}(1+\e);\e) + W(\frac{x}{2};\e) \Bigr) \bigg| \\
& \leq e^{-\frac{x}{2}} \biggl( (\e+h+\eta)M_1xe^{-\frac12\delta x} + xe^{-\frac{x}{2}\e}(\e+h+\eta)Me^{-\frac12\delta x} + \e x S_{h_0} \\
& \qquad\qquad +\frac{\e x}{h}\int_0^1 \textstyle\big| W_x(\frac{x}{2}(1+t\e);\e+h) - W_x(\frac{x}{2}(1+t\e);\e)\big| \de t \biggr) \\
& \leq (\e_1+\eta_1)C (1+x S_{h_0})e^{-\frac{x}{2}}
\end{align*}
where we used in particular \eqref{gammatoinf21} and the definition of $S_{h_0}$.
We now look at the term $A_3$:
\begin{align*}
A_3 & \leq 2(\e+h+\eta)Me^{-\frac12\delta x(1+\e)} \frac{ \big| W(\textstyle\frac{x}{2}(1+\e+h);\e+h) - W(\frac{x}{2}(1+\e);\e) \big|}{h} \\
& \leq (\e_1+\eta_1)^2MM_1xe^{- \delta x} + 2(\e_1+\eta_1)MS_{h_0}e^{-\frac12\delta x} \\
& \leq (\e_1+\eta_1)C\bigl( 1 + S_{h_0}\bigr) e^{-\frac12\delta x} \,.
\end{align*}
Finally, for the term $A_4$ we have
\begin{align*}
A_4 \leq 2\eta \Bigl( e^{-x} + (\e+h+\eta)Me^{-\delta x}\Bigr) \frac{|W(x;\e+h)-W(x;\e)|}{h} \leq (\e_1+\eta_1)C S_{h_0}e^{-\delta x}\,.
\end{align*}
By collecting all the previous bounds we find for all $h_0<h<\frac{\e_1}{2}$
\begin{equation} \label{gammatoinf24}
\bigg| \frac{\rho(x,\e+h)-\rho(x,\e)}{h} \bigg| \leq C \Bigl( 1 + (\e_1+\eta_1)S_{h_0} \Bigr) e^{-\frac12\delta x}\,.
\end{equation}
Inserting this estimate in \eqref{gammatoinf23} and using also \eqref{linearbounds1}--\eqref{linearbounds2} we obtain
\begin{align*}
\frac{|W(x;\e+h)-W(x;\e)|}{h}
& \leq C_0C\bigl( 1 + (\e_1+\eta_1)S_{h_0}\bigr) \biggl( \int_{x}^\infty e^{-\frac12\delta\xi}\de\xi + \int_0^x e^{\beta(\xi-x)}e^{-\frac12\delta\xi}\de\xi\biggr) \\
& \leq C\bigl( 1 + (\e_1+\eta_1)S_{h_0}\bigr)e^{-\frac12\delta x}\,.
\end{align*}
Similarly, using the expression \eqref{gammatoinfC1} for the derivative of $W$ we obtain
\begin{align*}
\frac{|W_x(x;\e+h)-W_x(x;\e)|}{h}
& \leq \frac{|\rho(x,\e+h)-\rho(x,\e)|}{h} \\
& \qquad + 2 \int_0^x e^{\xi-x} |G({\textstyle\frac{x}{2}},\xi)| \frac{|\rho(\xi,\e+h)-\rho(\xi,\e)|}{h} \de\xi \\
& \leq C\bigl( 1 + (\e_1+\eta_1)S_{h_0}\bigr)e^{-\frac12\delta x}\,.
\end{align*}
Therefore, adding up these two inequalities, and taking the supremum over $h\in(h_0,\frac12\e_1)$ and $x>0$, we end up with the bound $S_{h_0} \leq C\bigl( 1 + (\e_1+\eta_1)S_{h_0}\bigr)$; in turn, by choosing $\e_1$ and $\eta_1$ small enough, this yields $S_{h_0}\leq C$ for a uniform constant $C>0$ independent of $\e$, $\eta$, and $h_0$. As $h_0$ is arbitrary, we can conclude that $W$ and $W_x$ are uniformly Lipschitz continuous with respect to the variable $\e$ for $(x,\e,\eta)\in[0,\infty)\times[0,\frac{\e_1}{2})\times[0,\frac{\eta_1}{2})$. The proof of the Lipschitz continuity in the variable $\eta$ can be proved by a similar argument.

\medskip
\noindent\textit{Step 2: differentiability of $W$ in $\e$ and $\eta$.}
From the Lipschitz continuity it follows that $W$ is differentiable almost everywhere in the domain $[0,\infty)\times[0,\frac{\e_1}{2})\times[0,\frac{\eta_1}{2})$ with respect to the variables $\e$, $\eta$. We want  to prove now the continuity of $W_\e$ and $W_\eta$. As before we present the argument only for $W_\e$ (the one for $W_\eta$ being similar), and since $\eta$ takes always a fixed value we will not indicate the dependence on this variable to lighten the notation.

Notice that for all $x>0$ and $\e,\bar{\e}\in[0,\frac{\e_1}{2})$ we have the uniform estimates
\begin{equation} \label{gammatoinf25}
|W(x;\e)-W(x;\bar{\e})| \leq C|\e-\bar{\e}|e^{-\frac12\delta x}\,,
\qquad
|W_x(x;\e)-W_x(x;\bar{\e})| \leq C|\e-\bar{\e}|e^{-\frac12\delta x}\,,
\end{equation}
which in turn yield the bound on the partial derivative
\begin{equation} \label{gammatoinf25b}
|W_\e(x;\e)| \leq C e^{-\frac12\delta x}\,.
\end{equation}
For any two values $\e,\bar{\e}\in[0,\frac{\e_1}{2})$ we define the quantity
\begin{equation*}
S_{\e,\bar{\e}} := \sup_{x>0} \, |W_\e(x;\e)-W_\e(x;\bar{\e})|\,.
\end{equation*}

We first differentiate the function $\rho(x,\e)$, introduced in \eqref{gammatoinf22b}, with respect to $\e$: it is convenient to split its derivative into three parts,
\begin{equation*}
\rho_\e(x,\e) = \rho_\e^1(x,\e) + \rho_\e^2(x,\e) + \rho_\e^3(x,\e)\,,
\end{equation*}
with
\begin{align*}
\rho_\e^1(x,\e) := 2e^{-\frac{x}{2}} W_\e(\textstyle\frac{x}{2};\e) - 2e^{-\frac{x}{2}(1+\e)} W_\e(\frac{x}{2}(1+\e);\e) -2 W(\textstyle\frac{x}{2}(1+\e);\e) W_\e(\frac{x}{2}(1+\e);\e) \,,
\end{align*}
\begin{align*}
\rho_\e^2(x,\e) &:= xe^{-x(1+\e)} + xe^{-\frac{x}{2}(1+\e)} W(\textstyle\frac{x}{2}(1+\e);\e) - xe^{-\frac{x}{2}(1+\e)}W_x(\frac{x}{2}(1+\e);\e) \\
& \qquad\qquad -xW(\textstyle\frac{x}{2}(1+\e);\e)W_x(\frac{x}{2}(1+\e);\e)\,,
\end{align*}
\begin{align*}
\rho_\e^3(x,\e) := 2\eta\bigl(e^{-x} + W(x;\e)\bigr) W_\e(x;\e)\,.
\end{align*}
Notice that in view of \eqref{gammatoinf21} and \eqref{gammatoinf25b} we have a uniform bound $|\rho_\e(x,\e)|\leq Ce^{-\frac12\delta x}$. Therefore  we can differentiate under the integral sign in the equation \eqref{gammatoinf3}:
\begin{equation} \label{gammatoinf26}
\begin{split}
W_\e(x;\e) &= \sum_{i=1}^3 \biggl( -\int_x^\infty e^\xi Q(\xi) \rho_\e^i(\xi,\e)\de\xi + \int_0^x e^{\xi-x}\widetilde{G}(x,\xi)\rho_\e^i(\xi,\e)\de\xi \biggr) \\
& =: \sum_{i=1}^3 \Bigl( -A_i(x,\e) + B_i(x,\e) \Bigr)\,.
\end{split}
\end{equation}
The goal is now to obtain a bound on the difference $|W_\e(x;\e)-W_\e(x;\bar{\e})|$, in order to show the continuity of $W_\e$. Hence we proceed by considering each term in the sum \eqref{gammatoinf26}.

By a change of variable we can write
\begin{align*}
A_1(x,\e)
&= 2\int_x^\infty e^\xi Q(\xi) e^{-\frac{\xi}{2}} W_\e({\textstyle\frac{\xi}{2}};\e)\de\xi
-\frac{2}{1+\e}\int_{x(1+\e)}^\infty e^{\frac{\xi}{1+\e}} Q({\textstyle\frac{\xi}{1+\e}}) e^{-\frac{\xi}{2}} W_\e({\textstyle\frac{\xi}{2}};\e)\de\xi \\
& \qquad -\frac{2}{1+\e}\int_{x(1+\e)}^\infty e^{\frac{\xi}{1+\e}} Q({\textstyle\frac{\xi}{1+\e}}) W({\textstyle\frac{\xi}{2}};\e) W_\e({\textstyle\frac{\xi}{2}};\e)\de\xi \\
& = 2\int_x^\infty \biggl( e^\xi Q(\xi) - \frac{e^{\frac{\xi}{1+\e}}}{1+\e}Q({\textstyle\frac{\xi}{1+\e}})\biggr) e^{-\frac{\xi}{2}}W_\e({\textstyle\frac{\xi}{2}};\e)\de\xi \\
& \qquad + \frac{2}{1+\e}\int_x^{x(1+\e)} e^{\frac{\xi}{1+\e}} Q({\textstyle\frac{\xi}{1+\e}}) e^{-\frac{\xi}{2}} W_\e({\textstyle\frac{\xi}{2}};\e)\de\xi \\
& \qquad -\frac{2}{1+\e}\int_{x(1+\e)}^\infty e^{\frac{\xi}{1+\e}} Q({\textstyle\frac{\xi}{1+\e}}) W({\textstyle\frac{\xi}{2}};\e) W_\e({\textstyle\frac{\xi}{2}};\e)\de\xi \\
& =: a(x,\e) + b(x,\e) + c(x,\e).
\end{align*}
In order to obtain an estimate for the difference $|A_1(x,\e)-A_1(x,\bar{\e})|$ we consider the three terms on the right-hand side of the previous equation separately. The following estimates are obtained by using the bound \eqref{gammatoinf25b} and the Lipschitz continuity of the function $Q$ (see \eqref{linearbounds1}). For the first term we have
\begin{align*}
|a(x,\e)-a(x,\bar{\e})|
& \leq 2\int_x^\infty \bigg| \frac{e^{\frac{\xi}{1+\e}}}{1+\e}Q({\textstyle\frac{\xi}{1+\e}}) - \frac{e^{\frac{\xi}{1+\bar\e}}}{1+\bar\e}Q({\textstyle\frac{\xi}{1+\bar\e}}) \bigg| e^{-\frac{\xi}{2}} |W_\e({\textstyle\frac{\xi}{2}};\e)| \de\xi \\
& \qquad + 2\int_x^\infty \bigg| e^\xi Q(\xi) - \frac{e^{\frac{\xi}{1+\bar\e}}}{1+\bar\e}Q({\textstyle\frac{\xi}{1+\bar\e}})\bigg| e^{-\frac{\xi}{2}} \big| W_\e({\textstyle\frac{\xi}{2}};\e) - W_\e({\textstyle\frac{\xi}{2}};\bar{\e}) \big| \de\xi \\
& \leq C|\e-\bar{\e}|\int_x^\infty e^{-\frac{\xi}{2}} e^{-\frac14\delta\xi}\de\xi + C\bar{\e}S_{\e,\bar{\e}} \int_x^\infty e^{-\frac{\xi}{2}}\de\xi \\
& \leq C|\e-\bar{\e}| + C\bar{\e}S_{\e,\bar{\e}}\,.
\end{align*}
For the second term we obtain
\begin{align*}
|b(x,\e)-b(x,\bar{\e})|
& \leq \bigg|  \frac{2}{1+\e}\int_x^{x(1+\e)} e^{\frac{\xi}{1+\e}} Q({\textstyle\frac{\xi}{1+\e}}) e^{-\frac{\xi}{2}} \Bigl( W_\e({\textstyle\frac{\xi}{2}};\e) - W_\e({\textstyle\frac{\xi}{2}},\bar{\e}) \Bigr) \de\xi \bigg| \\
& \qquad+ \bigg| \frac{2}{1+\e}\int_x^{x(1+\e)} e^{\frac{\xi}{1+\e}} Q({\textstyle\frac{\xi}{1+\e}}) e^{-\frac{\xi}{2}} W_\e({\textstyle\frac{\xi}{2}};\bar\e)\de\xi \\
& \qquad\qquad - \frac{2}{1+\bar\e}\int_x^{x(1+\bar\e)} e^{\frac{\xi}{1+\bar\e}} Q({\textstyle\frac{\xi}{1+\bar\e}}) e^{-\frac{\xi}{2}} W_\e({\textstyle\frac{\xi}{2}};\bar\e)\de\xi \bigg| \\
& \leq C S_{\e,\bar{\e}} \int_x^{x(1+\e)} e^{-\frac{\xi}{2}} \de\xi + C|\e-\bar{\e}| \leq C\e S_{\e,\bar{\e}} + C|\e-\bar{\e}|\,.
\end{align*}
For the third term we use also \eqref{gammatoinf21} and \eqref{gammatoinf25}:
\begin{align*}
|c(x,\e)-c(x,\bar{\e})|
& \leq \bigg| \frac{2}{1+\e}\int_{x(1+\e)}^\infty e^{\frac{\xi}{1+\e}} Q({\textstyle\frac{\xi}{1+\e}}) W({\textstyle\frac{\xi}{2}};\e) \Bigl( W_\e({\textstyle\frac{\xi}{2}};\e) - W_\e({\textstyle\frac{\xi}{2}};\bar\e) \Bigr) \de\xi \bigg| \\
& \qquad + \bigg| \frac{2}{1+\e}\int_{x(1+\e)}^\infty e^{\frac{\xi}{1+\e}} Q({\textstyle\frac{\xi}{1+\e}}) W({\textstyle\frac{\xi}{2}};\e) W_\e({\textstyle\frac{\xi}{2}};\bar\e) \de\xi \\
& \qquad\qquad - \frac{2}{1+\bar\e}\int_{x(1+\bar\e)}^\infty e^{\frac{\xi}{1+\bar\e}} Q({\textstyle\frac{\xi}{1+\bar\e}}) W({\textstyle\frac{\xi}{2}};\bar\e) W_\e({\textstyle\frac{\xi}{2}};\bar\e) \de\xi \bigg| \\
& \leq C(\e+\eta) S_{\e,\bar{\e}} \int_{x(1+\e)}^\infty e^{-\frac12\delta\xi}\de\xi + C|\e-\bar{\e}| \leq C(\e+\eta)S_{\e,\bar{\e}} + C|\e-\bar{\e}|\,.
\end{align*}
Collecting all the previous estimates we get
\begin{equation} \label{gammatoinf26a}
|A_1(x,\e)-A_1(x,\bar{\e})| \leq C(\e_1+\eta_1)S_{\e,\bar{\e}} + C|\e-\bar{\e}|\,.
\end{equation}
Arguing in a completely similar way, using in particular the Lipschitz continuity of the function $\widetilde{G}$ (see \eqref{linearbounds2}), one can prove that
\begin{equation} \label{gammatoinf26b}
|B_1(x,\e)-B_1(x,\bar{\e})| \leq C(\e_1+\eta_1)S_{\e,\bar{\e}} + C|\e-\bar{\e}|\,.
\end{equation}

Observe that, as the functions $W$ and $W_x$ are continuous in both variables $(x,\e)$, the second term $\rho_\e^2$ is continuous in $\e$; this information, combined with the bound $|\rho^2_\e(x,\e)|\leq Ce^{-\frac12\delta x}$, yields by Lebesgue's Dominated Convergence Theorem
\begin{align} \label{gammatoinf26c}
|A_2(x,\e)-A_2(x,\bar{\e})| \leq \omega(\e-\bar{\e})\,,
\qquad
|B_2(x,\e)-B_2(x,\bar{\e})| \leq \omega(\e-\bar{\e})\,,
\end{align}
where $\omega(\e-\bar{\e})\to0$ as $\e-\bar{\e}\to0$ (uniformly with respect to $x$).

For the terms containing $\rho_\e^3$ we have by using the bounds \eqref{gammatoinf25}--\eqref{gammatoinf25b}
\begin{align} \label{gammatoinf26d}
| A_3(x,\e)-A_3(x,\bar{\e}) |
& \leq 2\eta\bigg| \int_x^\infty e^\xi Q(\xi) \bigl( W(\xi;\e)-W(\xi;\bar{\e}) \bigr)W_\e(\xi;\e) \de\xi \bigg| \nonumber \\
& \qquad + 2\eta\bigg| \int_x^\infty \e^\xi Q(\xi) \bigl(e^{-\xi} + W(\xi;\bar\e)\bigr) \bigl(W_\e(\xi;\e) - W_\e(\xi;\bar{\e})\bigr) \de\xi \bigg| \nonumber \\
& \leq C\eta|\e-\bar{\e}|\int_x^\infty e^{-\delta\xi}\de\xi + C \eta S_{\e,\bar{\e}} \int_x^\infty e^{-\delta\xi}\de\xi \nonumber \\
& \leq C\eta |\e-\bar{\e}| + C\eta S_{\e,\bar{\e}}\,,
\end{align}
and similarly for $B_3$
\begin{align} \label{gammatoinf26e}
| B_3(x,\e)-B_3(x,\bar{\e}) | \leq C\eta |\e-\bar{\e}| + C\eta S_{\e,\bar{\e}}\,.
\end{align}

By collecting \eqref{gammatoinf26a}--\eqref{gammatoinf26e} and inserting the bounds in \eqref{gammatoinf26} we have
\begin{equation*}
|W_\e(x,\e)-W_\e(x,\bar{\e})| \leq C(\e_1+\eta_1)S_{\e,\bar{\e}} + C|\e-\bar{\e}| + 2\omega(\e-\bar{\e})\,.
\end{equation*}
Therefore taking the supremum over $x>0$ and choosing smaller $\e_1$, $\eta_1$ if necessary we eventually obtain
\begin{equation*}
S_{\e,\bar{\e}} \leq C \Bigl( |\e-\bar{\e}| + \omega(\e-\bar{\e}) \Bigr) \qquad\text{for all }\e,\bar{\e}\in[0,\textstyle\frac{\e_1}{2}).
\end{equation*}
This yields the continuity with respect to the variable $\e$ of the partial derivative $W_\e$. In a similar fashion one can prove the continuity of $W_\e$ with respect to $\eta$, and also the continuity of $W_\eta$; we omit the details.

\medskip
\noindent\textit{Step 3: differentiability of $f$.}
Recall the definition \eqref{gammatoinf2} of $F$, the function $f(\e,\eta)$ can be written as follows:
\begin{align*}
f(\e,\eta) 
&= F(W(\cdot;\e,\eta),\e,\eta) \\
& = \int_0^\infty \textstyle e^\xi Q(\xi) \Bigl( e^{-\xi} - e^{-\xi(1+\e)} + 2e^{-\frac{\xi}{2}}W(\frac{\xi}{2};\e,\eta) -2e^{-\frac{\xi}{2}(1+\e)} W(\frac{\xi}{2}(1+\e);\e,\eta) \\
& \qquad\qquad\textstyle - \bigl( W(\frac{\xi}{2}(1+\e);\e,\eta) \bigr)^2 + \eta\bigl( e^{-\xi} + W(\xi;\e,\eta) \bigr)^2 \Bigr) \de\xi \,.
\end{align*}
Since $W\in C^1([0,\infty)\times[0,\frac{\e_1}{2})\times[0,\frac{\eta_1}{2}))$ and $W$ and its partial derivatives decay exponentially as $x\to\infty$, we can differentiate under the integral sign; the resulting expression is continuous in $(\e,\eta)$, and we conclude that $f$ is continuously differentiable in $[0,\frac{\e_1}{2})\times[0,\frac{\eta_1}{2})$.

\medskip
\noindent\textit{Step 4: Implicit Function Theorem.}
To conclude the proof it only remains to show the assumption in Lemma~\ref{lem:one-sideIFT} on the sign of the partial derivatives of $f$. We find by a straightforward computation, recalling that $W(\cdot;0,0)\equiv0$,
\begin{align*}
\frac{\partial f}{\partial \e}(0,0)
&= \int_0^\infty \xi Q(\xi) \de\xi\,.
\end{align*}
This integral has already been encountered in the proof of Proposition~\ref{prop:oscill}, where it was shown that it is strictly positive: see \eqref{oscill5}. For the partial derivative with respect to the variable $\eta$ we find similarly
\begin{align*}
\frac{\partial f}{\partial\eta}(0,0) = \int_0^\infty e^{-\xi}Q(\xi)\de\xi \,.
\end{align*}
To determine the sign of this integral, recall that the function $Q$ was defined in \eqref{linearG} as an alternating series $Q(\xi)=\sum_{n=0}^\infty (-1)^na_n e^{-2^n\xi}$, $a_n>0$, with the sequence of the coefficients $(a_n)_n$ strictly decreasing starting from $n=2$. We therefore have
\begin{align*}
\frac{\partial f}{\partial\eta}(0,0) = \sum_{n=0}^\infty (-1)^n a_n \int_0^\infty e^{-(1+2^n)\xi}\de\xi = \frac12 + \sum_{n=1}^\infty \frac{(-1)^n 2^{2n}}{(1+2^n)\prod_{j=1}^n(2^j-1)}\,.
\end{align*}
We can directly check that the previous quantity is strictly negative: one can just compute the contribution of the first five term, the rest being negative thanks to the monotonicity of the coefficients. Therefore $\frac{\partial f}{\partial \eta}<0$ and all the assumptions needed to apply Lemma~\ref{lem:one-sideIFT} are satisfied.
\end{proof}

\begin{lemma}[One-sided Implicit Function Theorem] \label{lem:one-sideIFT}
	Let $f:[0,x_0)\times[0,y_0)\to\R$ be continuously differentiable and assume that
	\begin{equation*}
	f(0,0)=0,\qquad \frac{\partial f}{\partial y}(0,0)>0,\qquad \frac{\partial f}{\partial x}(0,0)<0\,.
	\end{equation*}
	Then there exist $x_1\in(0,x_0)$ and $y_1\in(0,y_0)$ and a $C^1$-map $g:[0,x_1)\to[0,y_1)$ such that $g(0)=0$ and
	\begin{equation} \label{IFT}
	f(x,y)=0 \text{ with }(x,y)\in[0,x_1)\times[0,y_1) \qquad\text{ if and only if }\qquad y=g(x).
	\end{equation}
\end{lemma}

\begin{proof}
	This is just a small variation of the proof of the standard Implicit Function Theorem. The function $y\mapsto f(0,y)$ is strictly increasing in a neighborhood of zero and takes the value zero at $y=0$, therefore $f(0,y_1)>0$ if $y_1>0$ is small enough. By continuity we can find $\delta_1>0$ such that $f(x,y_1)>0$ for every $x\in(0,\delta_1)$. The function $x\mapsto f(x,0)$ is strictly decreasing in a neighborhood of zero and takes the value zero at $x=0$, therefore we can find $\delta_2>0$ such that $f(x,0)<0$ for every $x\in(0,\delta_2)$. Let $x_1:=\min\{\delta_1,\delta_2\}$. By reducing $x_1$ and $y_1$ if necessary we can assume that $\frac{\partial f}{\partial y}>0$ in the whole square $[0,x_1]\times[0,y_1]$. Then
	$$
	f(x,0)<0,\quad f(x,y_1)>0,\quad \frac{\partial f}{\partial y}(x,y)>0
	$$
	for every $(x,y)\in(0,x_1)\times(0,y_1)$. By continuity and the intermediate value theorem we obtain the existence of a function $g$ satisfying \eqref{IFT}. The differentiability of $g$ can be shown as in the proof of the standard Implicit Function Theorem.
\end{proof}

\begin{proof}[Proof of Theorem~\ref{thm:gammatoinf}]
	For every $\eta\in(0,\eta_1)$ we have obtained in Lemma~\ref{lem:gammatoinfIFT} a function $W(x;\e(\eta),\eta)$ satisfying the integral equation \eqref{gammatoinf3} together with the condition \eqref{gammatoinf3b}. Therefore by the construction discussed at the beginning of this section we have that the function $h(x)=e^{-x} + W(x;\e(\eta),\eta)$ solves the starting equation \eqref{gammatoinf0} for the values $(\e(\eta),\eta)$ of the two parameters, and thanks to \eqref{gammatoinf4bis} it decays exponentially at infinity:
	\begin{equation} \label{gammatoinf31}
	|h(x)| \leq Ce^{-\delta x}\,.
	\end{equation}
	
	It only remains to check the positivity of this solution. Assume by contradiction that $h(x_0)<0$ for some $x_0>0$. It is easily checked using the equation that $h$ has to remain negative also for larger values of $x$. Then for $x>x_0$
	\begin{equation*}
	h'(x)  = -\bigl(h\bigl(\textstyle \frac{x}{2}(1+\e)\bigr)\bigr)^2 + \eta\bigl(h(x)\bigr)^2 \leq \eta\bigl(h(x)\bigr)^2\,,
	\end{equation*}
	that is, $(\frac{1}{h(x)})' \geq -\eta$. By integration we end up with the inequality $\frac{1}{h(x)}\geq\frac{1}{h(x_0)}-\eta(x-x_0)$, or equivalently (since $h(x)<0$)
	\begin{equation*}
	h(x) \leq \frac{1}{\frac{1}{h(x_0)}-\eta(x-x_0)} \qquad\text{for } x>x_0.
	\end{equation*}
	But this is incompatible with \eqref{gammatoinf31} and therefore $h$ remains positive for all $x>0$.
\end{proof}


\appendix
\section{Fundamental solution to a linear delay equation} \label{sect:linear}

We construct in this section the fundamental solution to the delay equation 
\begin{equation*}
	\vphi'(x) = \vphi(x) - 2\vphi(x/2) \quad\text{for }x\in(0,\infty),\qquad \vphi(0)=0,
\end{equation*}
which appears in the previous sections as the linearization of \eqref{equation2} in the regime $\gamma\to\infty$.

\begin{lemma}[Fundamental solution]\label{lem:linearproblem}
The solution $G(x,\xi)$ solution to the linear delay equation, with a point source in $\xi>0$,
\begin{equation} \label{linearproblem}
\vphi'(x) = \vphi(x) - 2\vphi(x/2) + \delta(x-\xi)\quad\text{for }x\in(0,\infty),\qquad \vphi(0)=0,
\end{equation}
satisfies $G(x,\xi)=0$ for $x<\xi$, and has the representation
\begin{equation}\label{linearG}
G(x,\xi) =
e^x Q(\xi) + \widetilde{G}(x,\xi)\,,
\qquad\text{with }\quad Q(\xi) := e^{-\xi} + \sum_{n=1}^\infty \frac{(-1)^n2^{2n}e^{-2^n\xi}}{\prod_{j=1}^n (2^j-1)}\,,
\end{equation}
for $x>\xi$. The function $Q$ satisfies the bounds, for every $\xi>0$ and $\xi_2>\xi_1>0$,
\begin{equation} \label{linearbounds1}
|Q(\xi)| \leq C_0e^{-\xi}\,,
\qquad
|e^{\xi_1}Q(\xi_1)-e^{\xi_2}Q(\xi_2)| \leq C_0 e^{-\xi_1}|\xi_1-\xi_2|\,.
\end{equation}
Moreover for every $\beta\in(0,\frac12)$, $x>\xi$ and $x>\xi_2>\xi_1>0$ the function $\widetilde{G}$ satisfies the bounds
\begin{equation} \label{linearbounds2}
|\widetilde{G}(x,\xi)| \leq C_0e^{(1-\beta)(x-\xi)}\,,
\qquad 
|\widetilde{G}(x,\xi_1)-\widetilde{G}(x,\xi_2)| \leq C_0|\xi_1-\xi_2|e^{(1-\beta)(x-\xi_1)}\,,
\end{equation}
and finally for every $x_2>x_1>\xi$
\begin{equation} \label{linearbounds3}
|\widetilde{G}(x_1,\xi)-\widetilde{G}(x_2,\xi)| \leq C_0|x_1-x_2|e^{(1-\beta)(x_2-\xi)}\,.
\end{equation}
In all the previous estimates $C_0>0$ is a uniform constant depending only on $\beta$.
\end{lemma}

\begin{proof}
In order to prove the lemma, we will solve \eqref{linearproblem} explicitly via Laplace transform; in a second step we will obtain the representation \eqref{linearG} for $G$ by using the inverse Laplace transform and contour integration. Finally we will prove the estimates in the statement.

\medskip\noindent
\textit{Step 1.}
By taking the Laplace transform $\Phi(z):=\int_0^\infty \vphi(x)e^{-zx}\de x$ the equation \eqref{linearproblem} becomes
\begin{equation} \label{linear0}
(z-1)\Phi(z) = -4\Phi(2z) +e^{-z\xi}\,.
\end{equation}
We first construct the homogeneous solution $\Phi_h$ to \eqref{linear0} without the source term, that is, $(z-1)\Phi_h(z) = -4\Phi_h(2z)$.
By the rescaling $\Phi_h(z)=z^\alpha\Psi(z)$ with $\alpha=\frac{i\pi}{\ln2}-2$ (that is, $2^\alpha=-\frac14$), we reduce the problem to the simpler equation
\begin{equation} \label{linear1}
\Psi(2z)=(z-1)\Psi(z)\,,
\end{equation}
whose solution can be constructed as a power series $\Psi(z)=\sum_{n=-\infty}^\infty a_nz^n$. Indeed, the coefficients obey the recurrence relation $a_{n-1}=(2^n+1)a_n$, which gives up to a multiplicative constant (we choose $a_0=1$)
\begin{equation*}
\Psi(z) = 1+ \sum_{m=1}^\infty \frac{z^m}{\prod_{k=1}^m(2^k+1)} + \sum_{m=1}^\infty \Bigl( \prod_{k=0}^{m-1}(2^{-k}+1)\Bigr) z^{-m}\,.
\end{equation*}
The series is convergent for $|z|>1$. Notice also that $\Psi(z)=\frac{\Psi(2z)}{z-1}$ and hence $\Psi$ has a simple pole at $z=1$, since $\Psi$ is analytic in a neighborhood of $z=2$; it also follows that all the points $z=2^{-n}$, $n\in\N$, are poles of $\Psi(z)$.

We obtain the homogeneous solution $\Phi_h$ to \eqref{linear0} by multiplying $\Psi$ by $z^\alpha$.
We can now write the full solution to the problem \eqref{linear0}, including the source $e^{-z\xi}$, in terms of $\Phi_h$: indeed, by setting $\Phi(z)=\Phi_h(z)P(z)=z^\alpha\Psi(z)P(z)$ and plugging this ansatz into the equation \eqref{linear0}, we find (recall that $2^\alpha=-\frac14$ by the choice of $\alpha$)
\begin{equation*}
(z-1)z^\alpha\Psi(z)P(z) = z^\alpha\Psi(2z)P(2z) + e^{-z\xi}\,,
\end{equation*}
or equivalently, using \eqref{linear1},
\begin{equation*}
P(z) = P(2z) + \frac{e^{-z\xi}}{z^\alpha(z-1)\Psi(z)}\,.
\end{equation*}
By developing this relation and observing that $P$ tends to zero as $|z|\to\infty$ we find
\begin{equation*}
P(z) = \sum_{n=0}^\infty \frac{e^{-2^nz\xi}}{(2^nz-1)(2^nz)^\alpha\Psi(2^nz)}\,.
\end{equation*}
In turn, using the recurrence relation $\Psi(2^nz)=\prod_{k=0}^{n-1}(2^kz-1)\Psi(z)$, we obtain the following representation for the solution $\Phi$ to \eqref{linear0}:
\begin{equation} \label{linear1b}
\begin{split}
\Phi(z) &= z^\alpha\Psi(z)P(z)
= \Psi(z) \sum_{n=0}^\infty \frac{e^{-2^nz\xi}}{(2^nz-1)2^{n\alpha}\Psi(2^nz)}  \\
& = \sum_{n=0}^\infty \frac{2^{-n\alpha}e^{-2^nz\xi}}{\prod_{k=0}^{n}(2^kz-1)}
= \sum_{n=0}^\infty \frac{(-1)^n2^{2n}e^{-2^nz\xi}}{\prod_{k=0}^{n}(2^kz-1)}\,.
\end{split}
\end{equation}

\medskip\noindent
\textit{Step 2.}
We now use the inversion formula for the Laplace transform to obtain an expression for the solution to \eqref{linearproblem}. Observe that the function $\Phi$ is well-defined for $\Re(z)>1$ and the series in \eqref{linear1b} is uniformly convergent in $\{\Re(z)\geq1+\e\}$ for every $\e>0$, therefore we have
\begin{equation} \label{linear2}
G(x,\xi) = \frac{1}{2\pi i}\int_{L-i\infty}^{L+i\infty} e^{zx} \Phi(z)\de z
= \frac{1}{2\pi i}\sum_{n=0}^\infty \int_{L-i\infty}^{L+i\infty} \frac{(-1)^n2^{2n}e^{(x-2^n\xi)z}}{\prod_{k=0}^{n}(2^kz-1)} \de z\,,
\end{equation}
where the integral is on the vertical line $\{\Re(z)=L\}$, and $L$ is any real number larger than 1 (the integral here has to be interpreted in the sense of principal values, that is as $\lim_{R\to\infty}\int_{L-iR}^{L+iR}e^{xz}\Phi(z)\de z$, where the limit is in the $L^2$-sense). It is a standard exercise using contour integration on the right of the vertical line to show that $G(x,\xi)=0$ for every $x<\xi$.

In order to obtain the representation \eqref{linearG} for $x>\xi$, we ``move'' the vertical line from the position $L>1$ to a new position $\frac{1}{2}<\widetilde{L}<1$ by means of contour integration; more precisely, we integrate along the curve $\Gamma_R$ in the complex plane as in Figure~\ref{fig:contour}, where $R>0$ will be sent to infinity. The region enclosed by the curve contains only one simple pole at the point $z=1$. 
\begin{figure}
	\centering
	\includegraphics[scale=1]{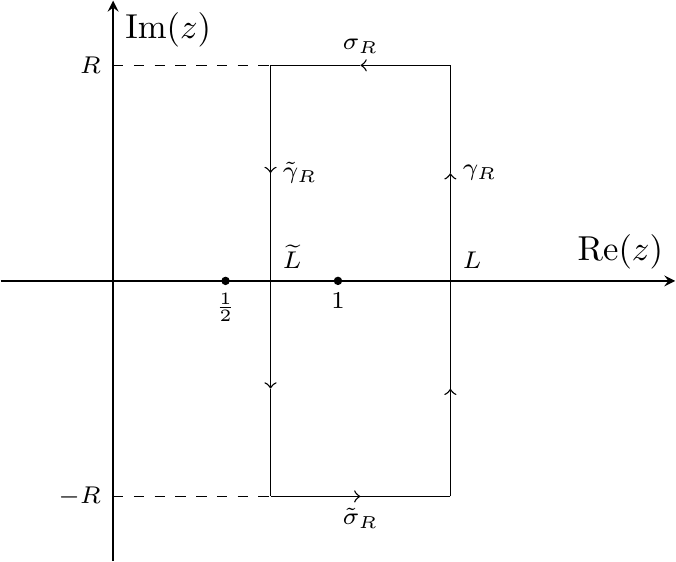}
	\caption{The curve $\Gamma_R$ used for contour integration.}
	\label{fig:contour}
\end{figure}
Notice that the integrals on the horizontal segments $\sigma_R$ and $\tilde{\sigma}_R$ vanish as $R\to\infty$: indeed (exchanging the sum and the integral, since the series is uniformly convergent on $\sigma_R$)
\begin{align*}
\bigg|\frac{1}{2\pi i} \int_{\sigma_R} \sum_{n=0}^\infty \frac{(-1)^n2^{2n}e^{(x-2^n\xi)z}}{\prod_{k=0}^n(2^kz-1)}\de z\bigg|
&=\bigg|\frac{1}{2\pi i} \sum_{n=0}^\infty \frac{(-1)^n2^{2n}}{\prod_{k=0}^n2^k}\int_{L}^{\widetilde{L}} \frac{e^{(x-2^n\xi)(t+iR)}}{\prod_{k=0}^n(t+iR-2^{-k})}\de t\bigg| \\
& \leq \frac{1}{2\pi}\sum_{n=0}^\infty \frac{2^{2n}}{2^{\frac{n(n+1)}{2}}} \int_{\widetilde{L}}^{L} \frac{e^{(x-\xi)t}}{R^{n+1}}\de t\to 0 \quad\text{as }R\to\infty\,,
\end{align*}
and similarly for the integral on $\tilde{\sigma}_R$. We therefore find using Cauchy's residual theorem
\begin{align*}
G(x,\xi) &= \frac{1}{2\pi i} \sum_{n=0}^\infty \int_{L-i\infty}^{L+i\infty} \frac{(-1)^n2^{2n}e^{(x-2^n\xi)z}}{\prod_{k=0}^{n}2^k \prod_{k=0}^n(z-2^{-k})}\de z \\
& = \sum_{n=0}^\infty \frac{(-1)^n2^{2n}}{\prod_{k=0}^{n}2^k} \, \Res\biggl(\frac{e^{(x-2^n\xi)z}}{\prod_{j=0}^n(z-2^{-j})};z=1\biggr) + \widetilde{G}(x,\xi) \\
& = e^{x-\xi} + \sum_{n=1}^\infty \frac{(-1)^n2^{2n}}{\prod_{j=1}^{n}2^j} \frac{e^{(x-2^n\xi)}}{\prod_{j=1}^n(1-2^{-j})} + \widetilde{G}(x,\xi) \\
& = Q(\xi) e^x + \widetilde{G}(x,\xi)\,,
\end{align*}
where the function $Q$ is as in the statement and the remainder $\widetilde{G}(x,\xi)$ is given by
\begin{equation} \label{linear3}
\widetilde{G}(x,\xi) = \frac{1}{2\pi i}\sum_{n=0}^\infty \int_{\widetilde{L}-i\infty}^{\widetilde{L}+i\infty} \frac{(-1)^n2^{2n} e^{(x-2^n\xi)z}}{\prod_{j=0}^n 2^j \prod_{j=0}^n(z-2^{-j})}\de z\,.
\end{equation}
Therefore the function $G(x,\xi)$ defined in \eqref{linearG} is the sought solution to \eqref{linearproblem}. The bounds \eqref{linearbounds1} for $Q$ follow directly by its explicit expression; in particular, for the second one we have
\begin{align*}
|e^{\xi_1}Q(\xi_1)-e^{\xi_2}Q(\xi_2)|
& = \sum_{n=1}^\infty\frac{2^{2n}}{\prod_{j=1}^n(2^j-1)} \big| e^{(1-2^n)\xi_1}-e^{(1-2^n)\xi_2}\big| \\
& \leq e^{-\xi_1}|\xi_2-\xi_1| \sum_{n=1}^\infty\frac{2^{2n}(2^n-1)}{\prod_{j=1}^n(2^j-1)} \leq C_0e^{-\xi_1}|\xi_2-\xi_1| \,,
\end{align*}
since the series is convergent.

\medskip\noindent
\textit{Step 3.}
We now turn to the proof of the estimates \eqref{linearbounds2}--\eqref{linearbounds3} involving $\widetilde{G}$. First notice that for $x>\xi$ the sum \eqref{linear3} defining $\widetilde{G}$ actually starts from $n=1$: indeed, for the term with $n=0$ we have
\begin{equation*}
\frac{1}{2\pi i} \int_{\widetilde{L}-i\infty}^{\widetilde{L}+i\infty} \frac{e^{(x-\xi)z}}{z-1}\de z=0\,,
\end{equation*}
as follows by a standard computation using contour integration on a half-circle on the left of the vertical line (observe that there are no poles inside the region of integration), and sending its radius to infinity. Therefore
\begin{equation} \label{linear5}
\widetilde{G}(x,\xi) = \sum_{n=1}^\infty \frac{(-1)^n2^{2n}}{2^{\frac{n(n+1)}{2}}}\frac{1}{2\pi i} \int_{\widetilde{L}-i\infty}^{\widetilde{L}+i\infty} \frac{e^{(x-2^n\xi)z}}{\prod_{j=0}^n(z-2^{-j})}\de z\,.
\end{equation}

By using the expression \eqref{linear5} it is straightforward to obtain the first bound in \eqref{linearbounds2}, with $\beta=1-\widetilde{L}$:
\begin{align*}
|\widetilde{G}(x,\xi)|
& \leq \frac{e^{\widetilde{L}(x-\xi)}}{2\pi} \sum_{n=1}^\infty \frac{2^{2n}}{2^{\frac{n(n+1)}{2}}} \int_{-\infty}^{\infty} \frac{\de t}{\prod_{j=0}^n|\widetilde{L}+it-2^{-j}|} \\
& \leq \frac{e^{\widetilde{L}(x-\xi)}}{2\pi} \sum_{n=1}^\infty \frac{2^{2n}}{2^{\frac{n(n+1)}{2}}} \int_{-\infty}^{\infty} \frac{\de t}{\bigl((\widetilde{L}-1)^2 + t^2\bigr)^\frac12 \bigl((\widetilde{L}-\frac12)^2 + t^2\bigr)^{\frac{n}{2}}}
\leq C_0e^{\widetilde{L}(x-\xi)}\,.
\end{align*}

For the second bound in \eqref{linearbounds2}, we have to isolate the term with $n=1$ in the expression \eqref{linear5} of $\widetilde{G}$. By contour integration on a large half-circle on the right or on the left of the vertical line $\{\Re (z) = \widetilde{L}\}$ (depending on whether $x<2\xi$ or $x>2\xi$), one can show that
\begin{equation} \label{linear6}
-\frac{1}{2\pi i}\int_{\widetilde{L}-i\infty}^{\widetilde{L}+i\infty} \frac{2e^{(x-2\xi)z}}{(z-1)(z-\frac12)}\de z =
\begin{cases}
4e^{x-2\xi} &\text{if }\xi<x<2\xi,\\
4e^{\frac12 x-\xi} &\text{if }x>2\xi.\\
\end{cases}
\end{equation}
By using this formula it is straightforward to check that for every $x>\xi_2>\xi_1>0$
\begin{equation*}
\bigg| \frac{1}{2\pi i}\int_{\widetilde{L}-i\infty}^{\widetilde{L}+i\infty} \frac{2e^{(x-2\xi_1)z}}{(z-1)(z-\frac12)}\de z - \frac{1}{2\pi i}\int_{\widetilde{L}-i\infty}^{\widetilde{L}+i\infty} \frac{2e^{(x-2\xi_2)z}}{(z-1)(z-\frac12)}\de z\bigg| \leq 8|\xi_1-\xi_2| e^{\widetilde{L}(x-\xi_1)}\,.
\end{equation*}
For all the other terms in the series \eqref{linear5} we have instead
\begin{align*}
\bigg| \sum_{n=2}^\infty \frac{(-1)^n2^{2n}}{2^{\frac{n(n+1)}{2}}} & \frac{1}{2\pi i} \int_{\widetilde{L}-i\infty}^{\widetilde{L}+i\infty} \frac{e^{zx}\bigl(e^{-2^n\xi_1z}-e^{-2^n\xi_2z}\bigr)}{\prod_{j=0}^n(z-2^{-j})}\de z \bigg| \\
& \leq \frac{e^{\widetilde{L}x}}{2\pi} \sum_{n=2}^\infty \frac{2^{2n}}{2^{\frac{n(n+1)}{2}}} \int_{-\infty}^{\infty} \frac{e^{-2^n\widetilde{L}\xi_1}|e^{2^n(\xi_1-\xi_2)(\widetilde{L}+it)}-1|}{\prod_{j=0}^n|\widetilde{L}+it-2^{-j}|} \de t\\
& \leq \frac{e^{\widetilde{L}(x-\xi_1)}|\xi_1-\xi_2|}{2\pi} \sum_{n=2}^\infty \frac{2^{3n}}{2^{\frac{n(n+1)}{2}}} \int_{-\infty}^{\infty} \frac{\bigl(\widetilde{L}^2+t^2\bigr)^\frac12}{\bigl((\widetilde{L}-1)^2 + t^2\bigr)^\frac12 \bigl((\widetilde{L}-\frac12)^2 + t^2\bigr)^{\frac{n}{2}}} \de t \\
& \leq \frac{e^{\widetilde{L}(x-\xi_1)}|\xi_1-\xi_2|}{2\pi}  \sum_{n=2}^\infty \frac{2^{3n}}{2^{\frac{n(n+1)}{2}}} \int_{-\infty}^{\infty} \frac{2\de t}{\bigl((\widetilde{L}-\frac12)^2 + t^2\bigr)^{\frac{n}{2}}}
\leq C_0|\xi_1-\xi_2|e^{\widetilde{L}(x-\xi_1)}
\end{align*}
for a constant $C_0$ depending only on the choice of $\widetilde{L}$. This completes the proof of the second inequality in \eqref{linearbounds2}. Finally, the bound \eqref{linearbounds3} follows by an analogous argument.
\end{proof}


\bigskip
\bigskip
\noindent
{\bf Acknowledgments.}
The authors acknowledge support through the CRC 1060 \textit{The mathematics of emergent effects} at the University of Bonn that is funded through the German Science Foundation (DFG).

\bibliographystyle{plain}
\bibliography{bibliography}

\end{document}